\documentclass[reqno,10pt]{amsart}
\usepackage{amscd,amssymb,amsmath,color,url,stmaryrd}
\usepackage{latexsym}
\usepackage[all]{xy}
\usepackage{defs}
\usepackage[colorlinks=true]{hyperref}
\usepackage{tikz-cd}
\usepackage[T1]{fontenc}

\usepackage{forest}

\def\Ig{{\rm Ig}}
\def\ss{{\rm ss}}

\def\brC{{\breve{C}}}
\def\brE{{\breve{E}}}

\def\NN{{\bbN}}
\def\GG{{\bbG}}
\def\QQ{{\bbQ}}

\def\ZZ{{\bbZ}}
\def\FF{{\bbF}}
\def\PP{{\bbP}}
\def\cO{{\calO}}
\def\cT{{\calT}}

\def\cX{{\calX}}
\def\cZ{{\calZ}}
\def\tilcX{{\wt{\calX}}}

\def\GL{{\rm GL}}
\def\SL{{\rm SL}}
\def\.{{,\dots,}}

\def\cI{{\calI}}
\def\cO{{\calO}}

\def\hatcE{{\widehat{\calE}}}
\def\cE{{\calE}}
\def\cP{{\calP}}

\def\cH{{\calH}}

\def\cM{{\calM}}
\def\ocH{{\overline{\cH}}}
\def\ocE{{\overline{\cE}}}
\def\ocP{{\overline{\cP}}}

\def\ocM{{\overline{\cM}}}

\begin{document}
\title{Torsion-stabilized modular curves of level $p$}
\author{Michael Temkin}

\thanks{This research was supported by IHES and MPIM-Bonn, a grant from the School of Mathematics at the IAS, and ISF Grant 1203/22}

\address{Einstein Institute of Mathematics\\
               The Hebrew University of Jerusalem\\
                Edmond J. Safra Campus, Giv'at Ram, Jerusalem, 91904, Israel}
\email{michael.temkin@mail.huji.ac.il}

\begin{abstract}
This is the first paper of a project on new integral models $\cX(N)$ of the modular curve $X(N)$. The final results for a general level $N$ will be obtained in the second paper, while this paper is devoted to giving all necessary background and definitions applicable to any $N$ and then working out the case of $\cX(p)$ with all possible details. We define $\cX(N)$ as the closure of $Y(N)$ in the space $\ocM_{1,N^2}=\ocM_{1,\Gamma}$, where $\Gamma=(\ZZ/N\ZZ)^2$, and show that for $N=p$ it is the blowup of the Katz-Mazur model $\tilcX(p)$ at all supersingular points, and hence $(\cX(p),Y(p))$ is the minimal toroidal resolution of $(\tilcX(p),Y(p))$. In fact, it is even log smooth over $(\ZZ,\ZZ[1/p])$, but this is special for the case when $p=N$. One can tautologically view $\cX(p)$ as the moduli space of $\Gamma$-stabilized genus-1 curves $(E,\Gamma)$ which can be smoothed to an elliptic curve labelled by its $N$-torsion, but our main results provide explicit criteria of the smoothability: $\cX(p)$ parameterizes $\Gamma$-equivariant stable genus-1 curves $(E,\Gamma)$ such that the action satisfies two explicit conditions formulated in the paper.
\end{abstract}

%\subjclass{}
\maketitle

\setcounter{tocdepth}{1}
\tableofcontents

\section{Introduction}

\subsection{Wild Hurwitz spaces}
Our study of new models of modular curves grew out from a wider class of problems that can be roughly characterized as the search for models and moduli spaces of Hurwitz spaces of covers in presence of wild ramification. In a sense, modular curves parameterize the basic abelian case when one classifies (rigidified) isogenies, and general wild Hurwitz spaces should be (substantially?) more complicated. Although the broader context is not necessary for reading this paper we prefer to start with the origin of the problem, a brief description of the state of the art in the field of wild Hurwitz spaces (which is only starting to be studied and is almost terra incognita when the wild ramification index is divisible by $p^2$), and some general ideas and approaches that originated there and also apply to the study of modular curves.

\subsubsection{The compactified Hurwitz space $\ocH_{2,4}$}
In \cite{AO} Abramovich and Oort constructed a natural compactification over $\ZZ$ of the classical Hurwitz space $\cH_{2,4}=\PP^1_{\ZZ[1/2]}$ which parameterizes double covers $E\to\PP^1$ ramified at 4 marked points. In fact, $\cH_{2,4}=\PP^1_\lam$ parameterizes Kummer covers $y^2=t(t-1)(t-\lam)$ and over the three cusps $0,1,\infty$ the cover degenerates to a cover of two $\PP^1$'s, each one containing two marked points, meeting at a node by a pair of two $\PP^1$'s meeting at two nodes.

Abramovich and Oort implemented an idea of Pandaripande to embed Hurwitz schemes into certain moduli spaces of stable maps, which are natural moduli spaces proper over $\ZZ$, and define a natural compactification just by taking the closure in this ambient space. In \cite{AO} the outcome in the simplest case of $\ocH=\ocH_{2,4}$ was computed explicitly. It turned out the fiber of $\ocH$ over $\FF_2$ consists of three $\PP^1_j$ lines, which parameterize the ordinary locus over which the four branch points over $\{0,1,\lam,\infty\}$ on $E_\QQ$ collide to two branch points on $E_{\FF_2}$, and a $\PP^1_\lam$ component, which parameterizes supersingular covers, in which all four points collide to the single ramification point. Of course, this component intersects the $\PP^1_j$'s at $\lam=0,1,\infty$, and the intersection points on $\PP^1_j$'s correspond to the only supersingular elliptic curve in characteristic 2, which has $j=0$.

We refer to \cite{AO} for details and only mention here that the embedding into the moduli space of stable maps with four marked points forces one to keep the ''ramification points'' disjoint even modulo 2. Their proof provides a natural stable map $\ocE\to\ocP$ between prestable $\ocH$-curves of genus 1 and 0, respectively, and $\ocE$ also comes equipped with four disjoint sections $P_0,P_1,P_\infty,P_\lam$. To a large extent the construction reduces to first extending the cover $\cE\to\cP$ over $\ZZ[1/2]$ to a finite cover of prestable curves (simultaneous semistable reduction) and then blowing up the locus over $\FF_2$ where the closure of the sections collide. Sometimes such a blowup also forces one to modify the base so that the curves stay flat over the base. The collision pattern is different in the ordinary and supersingular cases and this is reflected in the space $\ocH$.

\begin{rem}
In my view the decisive and most non-standard step in this approach is to ignore the fact that ramification patterns are different for tame and wild ramification, and to prevent the four points from colliding in characteristic 2 by using moduli spaces of stable marked curves. In a sense, one forces the wild ramification pattern to deform to a tame one already in characteristic 2, which is an easier task than studying an immediate deformation to mixed characteristic. As a result, a valuable information is added to the functor when $p$ is not invertible, and I view it as a new version of the wild ramification/level structures. This information possesses moduli of its own, so this approach naturally produces moduli spaces with reducible fibers over wild primes, which is usually crucial in order to obtain nice enough, e.g. log regular, models.
\end{rem}

\subsubsection{Modular interpretation}
Although an explicit description of the space $\ocH$ was obtained in \cite{AO}, it remained unclear if it possesses a natural modular interpretation, and this question was recently solved by Hippold in \cite{Hip} in the case of $\deg=p$ (i.e. a nice space of covers of degree $p$ is constructed over $\ZZ[\frac{ 1}{(p-1)!}]$), providing a first general definition which deals with the case when wild ramification happens. The main obstacle was that in the wild case, including already $\ocH_{2,4}$, there might be components in the $\FF_p$-fiber on which the cover is inseparable. The reduction loses information in such a case: not any such a cover can be lifted from characteristic $p$, and the location of the ramification divisor on inseparable components plays a crucial role and provides a new moduli over $\FF_p$. A central tool in \cite{Hip} is to enrich a functor further by assigning a meromorphic differential form to each inseparable component so that its zeros detect the specialization of the ramification divisor on the component.

\begin{rem}
\cite{Hip} continued the research of \cite{CTT} and \cite{BT} where the different and the stable reduction of covers of degree $p$ were constructed over a valuation ring. The differential form appeared in \cite{BT} as a scaled reduction of the different, and it also has an interpretation as Kato's refined Swan conductor. In fact, this series of papers in Berkovich geometry has started with the study of stable reduction of Kummer covers in the mixed characteristic -- another instance of the same problem as was studied by Abramovich-Oort.
\end{rem}

\subsubsection{Modular curves}
The example of Abramovich-Oort has also an interpretation from a different world: $\cH_{2,4}$ is just the modular curve $X(2)$ because giving such a cover with fixed ramification is the same as giving the (generalized) elliptic curve $E$ with a fixed 2-torsion subgroup $E[2]$. A simple computation shows that the compactification $\ocH$ is the blowup of the Katz-Mazur compactification $\tilcX(2)$ at the supersingular point. As we saw above, for $p=2$ this blowup has a natural modular interpretation, so one can wonder if this can be generalized for other $p$. It turns out that there exists a simple non-modular definition which follows the non-collision principle -- do not allow marked points in characteristic zero to collide modulo $p$, even if they want to. To achieve this we just define $\cX(N)$ to be the closure of $Y(N)$ in $\ocM_{1,N^2}=\ocM_{1,\Gamma}$, where $\Gamma=(\ZZ/N\ZZ)^2$.

The integral models $\cX(N)$ of the modular curves are not directly related to Hurwitz spaces of maps to $\PP^1$, though they are components of the Hurwitz spaces that parameterize isogenies of order $N^2$ between elliptic curves. We think that they are fascinating moduli spaces which can be viewed as moduli spaces of elliptic curves with wild level structure and (to the best of our knowledge) were not described earlier. We hope that the methods and ideas developed in this project will be also useful to study wild level structures on curves of higher genus and on abelian varieties, but of course one should first study the simplest cases of elliptic curves and their moduli spaces. In particular, in this case many additional tools (such as the well studied regular model of Katz-Mazur) are available. Even in this setting, we still decided to split this project to two papers. This paper is devoted to defining $\cX(N)$ and the objects it parameterizes -- prestable (but not semistable) $\Gamma$-equivariant elliptic curves stabilized by $\Gamma$ and various properties of the action needed to isolate $\cX(N)$ among the other components. We test these properties on the case of $N=p$, which is worked out with all details in this paper. In particular, we study what other components are, and explain why we need to impose various restrictions on the action to get rid of the non-smoothable components. This study is not needed to prove our main results about $\cX(p)$ but illustrates our choices and the wide context we are working with. In the sequel paper we will use the foundations developed here to study the case of an arbitrary $N$. Finally, I hope that quite analogous methods will also extend to level structures on abelian varieties, but at this stage it is a bit premature to discuss this.

\subsection{Overview of the paper}
We start with introducing the stacks $\cX(N)$ in \S\ref{sec2} as the closure of $Y(N)$ in $\ocM_{1,\Gamma}$ and observing that over $\ZZ[1/N]$ one just recovers the Deligne-Rapoport compactification. Then we provide in Lemma~\ref{tautology} the tautological interpretation of $\cX(N)$ as the moduli space of smoothable $\Gamma$-stabilized genus-1 curves $(E,\Gamma)$ and start to study the main question of the paper -- what is a constructible/explicit definition of such a functor? The general principle is very simple -- consider a smoothing, and study how all properties of the generic fiber $(E_\eta,\Gamma=E_\eta[N])$ restrict to the closed fiber, what are the restrictions they impose. We provide a whole list of such properties in \S\ref{redsec}
\begin{itemize}
\item[(1)] The $\Gamma$-action by translations is compatible with $\Pic^0$ and a global derivation.
\item[(2)] The isogeny of multiplication by $p$.
\item[(3)] A group structure.
\end{itemize} 
It turns out that versions of (1) and (2) suffice to cut out $\cX(p)$ (and also $\cX(N)$ as we will prove in the sequel paper), while the group structure is not needed and in fact can be reconstructed from (1).

Of course, the choice of properties that characterise smoothability can be non-unique and to some extent is a matter of taste or convenience, but at least we will show on examples that our list is minimal and no property can be removed, see Remarks~\ref{phirem} and \ref{flatrem}(ii). Note that $\Gamma$ acts on $\ocM_{1,\Gamma}$ by translating the marked points, and this fixes the points of $Y(N)$ because the action on $\Gamma=E[N]$ comes from the action on $E$ itself. It follows that $\cX(N)\subseteq(\ocM_{1,\Gamma})^\Gamma$ and hence $\cX(N)$ parameterizes $\Gamma$-equivariant $\Gamma$-labelled curves $(E,\Gamma)$, that are called {\em $\Gamma$-marked} in \S\ref{gammasec}. We view this condition as the most important rigidifying property, which provides a first approximation of smoothability.

One of the advantages of $\Gamma$-equivariance is its amenability for study by quite standard tools, and in \S\ref{sec3} we apply them to the space $(\ocM_{1,\Gamma})^\Gamma$ of $\Gamma$-marked curves when $N=p$. The situation over $\ZZ[1/p]$ is quite simple, so we mainly concentrate on studying the $\FF_p$-fiber. First, we classify the geometric points and show that on the set-theoretic level the smoothability imposes just one more restriction of supersingular nature beyond $\Gamma$-equivariance. Then we study deformations and show that there do exist non-smoothable deformations of smoothable $\Gamma$-marked curves. This provides concrete examples justifying two restrictions on the action we impose later to reduce $(\ocM_{1,\Gamma})^\Gamma$ further to $\cX(p)$. In particular, we interpret some non-smoothable deformations along the ordinary locus as those that point in the direction of non-semistable components of $(\ocM_{1,\Gamma})^\Gamma$ over $\ZZ[1/p]$. Such components did not appear in \cite{DR}, but in our setting they must be ruled out by an additional condition. We finish this section with an explicit description in \S\ref{fixedloci} of a few simpler fixed locus spaces, such as $(\ocM_{0,\oGamma_0})^{\Gamma_0}=\Spec(\ZZ[\xi_p])$, where $\Gamma_0=\FF_p$ acts on $\oGamma_0=\FF_p\cup\{\infty\}$ by translations. This geometrically explains some phenomena with non-liftable deformations in terms of ramification over $\ZZ$.

Main results are proved in \S\ref{4sec}. First, we introduce a condition $(\phi_p)$, which postulates existence of a morphism $\phi_p$ which imitates the restriction of the multiplication-by-$p$ isogeny from a smooth generic fiber. It is a bit subtle because $\phi_p$ reduces the level to $N/p$, but in the case of $N=p$ it can be reformulated just by saying that the wild level-$p$ structure $\Gamma\into E$ in our sense restricts to the Drinfeld's level structure on the semistable core of $E$. The obtained closed substack $(\ocM_{1,\Gamma})^{\Gamma,\phi_p}$ of $(\ocM_{1,\Gamma})^\Gamma$ already has the ``right'' set of geometric points, but still possesses non-smoothable deformations over $p$. Therefore we introduce another condition on the action, denoted ($\flat$) and called {\em affinity}, which imitates the fact that the action on a smooth generic fiber $E_\eta$ (in a smoothing) is by translations, and hence preserves the invariant differential form $\partial_\eta$ and the group $\Pic^0_{E_\eta}$. The second condition is straightforward, and only essential when $N\le 3$, while the first condition is critical. The restriction of $\partial_\eta$ onto an irreducible component $Z$ of the closed fiber can vanish, so for each $Z$ we rescale $\partial_\eta$ by some function, and then restrict to $Z$ obtaining a {\em scaled derivation} on $Z$. We show that $\partial_Z$ is unique up to a unit, and the affinity condition means that $\Gamma$ preserves it. The main result proved in Theorem~\ref{mainth} states that $\cX(p)=(\ocM_{1,\Gamma})^{\Gamma,\flat,\phi_p}$, this scheme is toroidal with respect to $Y(p)$ and can be obtained by blowing up the supersingular locus in $\tilcX(p)$.

The paper contains an appendix where some material about stable curves and their deformations is collected. All these facts are quite standard and well known to experts, and the only novelty is in the presentation of the stabilization (or folding) process and some terminology, including the notion of crown and the usage of the word ``folding'' in addition to the usual retraction/contraction. Our main goal here is to organize this background material in a form that will be convenient for referencing in this and sequel papers of the project. By saying ``folding'' we want to stress the point of view that a genus-1 pointed prestable $S$-curve ``knows'' about all potential torsion, and when working on the level $N$ we choose specific models by folding or unfolding components stabilized by an appropriate torsion.

\subsection{Conventions}\label{convsec}

\subsubsection{The base}
We will work over a qcqs $S$, which can be a scheme or a DM stack. Due to this, ``locally'' on $S$ will always mean Zariski or \'etale locally, respectively. When we say ``geometric'' points and ``\'etale'' covers this also includes the case of schemes. Of course, everything can be adapted further to Artin stacks and smooth topology, but we leave this to the interested reader.

\subsubsection{Prestable marked curves}
By a {\em prestable marked} curve over a base stack $S$ we mean a pair $(C_S,D_S)$, where $f\:C_S\to S$ is a proper nodal curve with geometrically connected fibers (i.e. $f_*\cO_{C_S}=\cO_S$) and $D_S\into C^\sm_S$ is a divisor which is disjoint from the nodal set and is \'etale over $S$. If one also fixes an isomorphism $D_S=\Gamma_S=\Gamma\times S$ for a finite set $\Gamma$, the curve is called {\em $\Gamma$-labelled}, and the term {\em $n$-pointed} is used when $\Gamma=\{1\. n\}$. When $S$ is clear from the context we may omit it from the notation: $(C,D)$, etc.

\subsection{Acknowledgements}
This research started quite unusually, at least, for my practice. I was giving a talk at the university of Toronto about the compactified Hurwitz spaces of $p$-covers based on the work of Abramovich--Oort, the works \cite{CTT} and \cite{BT} on the non-archimedean approach to similar questions and the recent work \cite{Hip} by Hippold. This was too much material which used too many different techniques, and the talk was a failure. Then I modified the talk by removing anything related to Berkovich geometry and suggesting a natural (but not historical) way to discover the compactification functor as a continuation of the work of Abramovich-Oort. This approach produced a successful talk at the University of Pennsylvania, and led me to realize soon after the talk that similar principles can be used to construct a natural new compactification/model of a modular curve and describe it in a functorial way. I am grateful to both institutions for the invitations. In addition, during my work when some details still were quite vague I discussed this construction with many experts, including Pierre Deligne, Bhargav Bhatt, Mingjia Zhang, Ilya Tyomkin, Ofer Gabber, Ahmed Abbes, Dustin Clausen, K\k{e}stutis \v{C}esnavi\v{c}ius, Peter Scholze, Alexander Beilinson, Kazuya Kato, Stefan Wewers, Florian Pop, and I am very grateful to them for their patience. Special thanks to Dan Abramovich for the careful reading of a draft version and pointing out numerous inaccuracies.

AI Gemini Pro was used to draw some diagrams faster and more impressively than I would be able to do myself.

\section{Torsion-stabilized models of modular curves}\label{sec2}

\subsection{The model $\cX(N)$}

\subsubsection{$\Gamma$-stable curves}
Fix an integer $N>1$ and the group $\Gamma=(\ZZ/N\ZZ)^2$ with generators $\gamma_1=(1,0)$ and $\gamma_2=(0,1)$. By a {\em $\Gamma$-pointed stable $S$-curve} or just a {\em $\Gamma$-stable $S$-curve} we mean a stable $S$-curve $(C,D)$ with a trivialization $D=\Gamma_S=\Gamma\times S$. The moduli space $\ocM_{g,\Gamma}$ of $\Gamma$-stable curves is a smooth and proper DM stack over $\ZZ$, which is canonically isomorphic to $\ocM_{g,N^2}$, once a bijection $\Gamma\toisom\{1\. N^2\}$ is fixed.

\subsubsection{The modular curves $Y(N)$ and $X(N)$}
The moduli space $Y(N)$ classifies elliptic curves with $\Gamma$-trivialized (or rigidified) $N$-torsion or just {\em $\Gamma$-leveled} elliptic curves. In other words, if $S$ is a scheme, then $Y(N)(S)$ is the groupoid of pairs $(E_S,\Gamma_S\toisom E_S[N])$, where $E_S$ is an elliptic $S$-curve. Such a set can be non-empty only when $N$ is invertible on $S$, so $Y(N)$ is a $\ZZ[1/N]$-scheme, and even a $\ZZ[1/N,\mu_N]$-scheme, as follows from the existence of the Weil pairing.

Deligne-Rapoport constructed in \cite{DR} a modular $\ZZ[1/N]$-compactification $X(N)$ which classifies $\Gamma$-leveled generalized elliptic curves. Its cusps correspond to $N$-gons of $\PP^1$'s.

\subsubsection{The model $\cX(N)$}
We can interpret an element $(E_S,\Gamma_S)\in Y(N)(S)$ as a $\Gamma$-stable elliptic $S$-curve, so a natural closed immersion $Y(N)\into\ocM_{1,\Gamma}$ arises. Then we define $\cX(N)$ to be the closure of $Y(N)$ in $\ocM_{1,\Gamma}$. In particular, $\cX=\cX(N)$ is equipped with the tautological family of $\Gamma$-stable genus-1 curves $(E_\cX,\Gamma_\cX)$.

\subsubsection{The vertical compactification}
Let us first study the classical part obtained by restriction to $\Spec(\ZZ[1/N])$. As one might expect, in this case we just obtain a new natural construction of $X(N)$ -- the Deligne-Rapoport compactification of $Y(N)$ over $\Spec(\ZZ[1/N])$.

\begin{lem}\label{Nlem}
Keep the above notation, then $X(N)=\cX(N)\otimes\ZZ[1/N]$.
\end{lem}
\begin{proof}
First, the moduli space $X=X(N)$ possesses a tautological family $(E_X,\Gamma_X)$ where $E_X$ is a generalized elliptic curve (a semistable 1-pointed genus-1 curve) and $\Gamma_X$ is its $N$-torsion. By the construction $\Gamma_X$ stabilizes $E_X$. So, the closed immersion $Y(N)\into \ocM_{1,\Gamma}$ extends to $X(N)\into\ocM_{1,\Gamma}$. Since $X(N)$ is a compactification of $Y(N)$ over $\Spec(\ZZ[1/N])$, we obtain that $X(N)=\cX(N)\otimes\ZZ[1/N]$.
\end{proof}

\subsubsection{The tautological functor}
Of course our construction provides a functorial description of $\cX(N)$. We say that a $\Gamma$-stable genus-1 $S$-curve $(E,\Gamma_S)$ is {\em smoothable by a $\Gamma$-leveled elliptic curve} if there exists an \'etale covering $\{S_i\}$ of $S$, closed immersions $S_i\into S'_i$ with integral $S'_i$ and $\Gamma$-stable genus-1 $S'_i$-curves $(E'_i,\Gamma_{S'_i})$ such that the restrictions onto $S_i$ are the pullbacks of $(E,\Gamma_S)$ and the generic fibers are $\Gamma$-labeled elliptic curves. When this cannot cause misunderstanding we will simply say {\em $\Gamma$-smoothable}.

\begin{lem}\label{tautology}
Let $S$ be a scheme and $\cX=\cX(N)$, then $\cX(S)$ is the groupoid of $\Gamma$-stable genus-1 $S$-curves smoothable by $\Gamma$-leveled elliptic curves.
\end{lem}
\begin{proof}
The family over $\cX$ is $\Gamma$-smoothable because its generic fiber is an elliptic curve with $\Gamma$-marked $N$-torsion. Conversely, if $(E,\Gamma_S)$ is a $\Gamma$-smoothable $\Gamma$-stable genus-1 $S$-curve, then we should prove that the induced morphism $S\to\ocM_{1,\Gamma}$ lands in $\cX$. This can be checked \'etale locally, so we can replace $S$ by $S_i$ as in the definition of $\Gamma$-smoothability. Then the map $S_i\to\ocM_{1,\Gamma}$ is the restriction of the map $S'_i\to\ocM_{1,\Gamma}$ induced by $(E_i,\Gamma_{S'_i})$, and this map lands in $\cX$ because the generic fiber of the latter family is an elliptic curve with $\Gamma$-marked $N$-torsion.
\end{proof}

\subsubsection{Reduction data}\label{redsec}
The tautological $\Gamma$-smoothability condition is not really constructible, so we use this description only for the sake of semantical convenience. The real task now is to express $\Gamma$-smoothability in concrete terms. A general principle is very simple -- restrict enough data/structure from a smooth generic fiber so that these data guarantee $\Gamma$-smoothability. Here are a few natural directions that will be used or discussed in the sequel.

\begin{itemize}
\item[(1)] {\em  $\Gamma$-action.} Since $\Gamma$ acts by translation on the generic fiber $(E_\eta,\Gamma_\eta)$ of $(E_{\cX},\Gamma_\cX)$ its action extends to the whole $(E_\cX,\Gamma_\cX)$. This $\Gamma$-equivariance is the most basic property which will be combined with subtler ones. We say that the action of $\Gamma$ on $E_\eta$ is {\em affine} referring to the action being by translation. This property can be also interpreted as the triviality of the induced action on $\Pic^0_{E_\eta}$, and it imposes two restrictions on the action on the closed fibers that we also call {\em affinity}:
\begin{itemize}
\item[(a)] {\em Invariance of scaled vector fields.} For each connected component $Z$ of a fiber $E_s$ the action of the stabilizer $\Gamma_Z$ preserves any vector field $\partial_Z$ obtained by extending, rescaling and restricting onto $Z$ of an invariant vector field on $E_\eta$. (There is also a dual construction with an invariant differential form, but it is more convenient in this case to work with the vector fields.)
\item[(b)] {\em Picard invariance.} The induced $\Gamma$-action on $\Pic^0_{E_\cX/\cX}$ is trivial because $\Pic^0_{E_\cX/\cX}$ is flat over $\cX$ and the action on the generic fiber is affine. In fact, this condition is subsumed in (a) when $N>3$, but is needed to eliminate exotic actions when $N=2,3$.
    \end{itemize}
\item[(2)] {\em Multiplication by $p$.} For each $p|N$ the multiplication by $p$ on $E_\eta$ lowers the level and hence induces a morphism of marked curves $(E_\eta,E_\eta[N])\to (E_\eta,E_\eta[N/p])$ rather than an endomorphism. Thus, a morphism $\phi_p\:(E_\cX,\Gamma)\to(E'_\cX,p\Gamma)$ arises, where $E'_\cX$ is the retraction of $E_\cX$ stabilized by $p\Gamma$. This morphism is compatible with $\Gamma$ and acts by multiplication by $p$ on $\Pic^0_{E_\cX/\cX}$ and $\Gamma_\cX$.
\item[(3)] {\em The group structure.} The group structure of $E_\eta$ extends to an $\cX$-group structure on the crown $\brE_\cX$ (or the smooth part of the bottom layer of $E_\cX$, see \S\ref{crownsec}).
\item[(4)] {\em Log structures.} The usual log structures of $\ocM_{1,\Gamma}$ and the universal curve over it can be restricted to $\cX$ and $E_\cX$.
\end{itemize}

We will see in \cite{X(N)} that the $\Gamma$-smoothability is already achieved by imposing conditions described in (1) and (2). In this paper we study the case of $N=p$, and then (2) reduces to a very simple property about retraction onto the core and its relation to Drinfeld level structure (which are introduced in \cite[Chapter~1]{KM}).

One can prove (this is not immediate and uses the $\Gamma$-action) that the group structure on the smooth generic fiber induces a group structure on the crown with an affine action $\brE\times E\to E$ on $E$. Such a structure extends generalized elliptic curves to the non-semistable case, and we suggest to call it a {\em folding elliptic curve}. In particular, because one can show that for any $\Gamma$-equivariant folding (or retraction) $E\to E'$, the target acquires a natural folding elliptic curve structure too, and the folding map $\brE(S)\to\brE'(S)$ is a homomorphism. We plan to study folding elliptic curves (without $\Gamma$-marked structure) in a separate paper since the group structure is not used in our approach to description of $X(p)$.

Finally, we also prefer not to involve the log structures in order to solve the problem within the category of schemes. I do expect, however, that the log structures provide a more conceptual interpretation of part of the data and their usage can shorten some arguments and replace a part of the affinity condition. A logarithmic description of $\cX(N)$ will be studied elsewhere.

\subsection{$\Gamma$-action}\label{gammasec}

\subsubsection{The action}
Note that $\cM_{1,\Gamma}$ is a scheme for $N>2$, but $\ocM_{1,\Gamma}$ is a stack. For example, there exists a $\Gamma$-stable curve with a genus-1 component having one node and no marked points, and such a component possesses non-trivial automorphisms fixing the node. The group $\Gamma$ acts on the whole stack $\ocM_{1,\Gamma}$ just by the action on the set of marked points. Intuitively, it is of a ``set theoretic nature'', so one does not have to use the general (and more subtle) theory of group actions on stacks. The way to formalize this is to observe that the morphism $\ocM_{1,\Gamma}\to\ocM_1$ is projective and the action is fiberwise (the action is $\Gamma$-equivariant once the target is provided with the trivial action). So, one can describe this action in terms of the usual descent and admissible $\Gamma$-actions on schemes.

\subsubsection{The fixed point locus}\label{fixedsec}
We can thus use descent to define the fixed locus $(\ocM_{1,\Gamma})^\Gamma$, and in our case it is a closed substack, as one should expect in the case of actions of ``set theoretic origin''. In particular, an $S$-point $(E_S,\Gamma_S)$ of $\ocM_{1,\Gamma}$ is $\Gamma$-fixed if and only if there exists an $S$-action of $\Gamma$ on $E_S$ which extends the translation action of $\Gamma$ on $\Gamma_S$.

\subsubsection{$\Gamma$-marked curves}
A $\Gamma$-equivariant $\Gamma$-stable $S$-curve $(E_S,\Gamma_S)$ will be called {\em $\Gamma$-marked}. As we saw, $(\ocM_{1,\Gamma})^\Gamma$ is the moduli space of such objects. In the case of an integral $S$, we say that $(E_S,\Gamma_S)$ is {\em generically smooth} if $(E_\eta,\Gamma_\eta)$ is a $\Gamma$-leveled elliptic curve. Thus, $\cX(N)$ is the substack of $(\ocM_{1,\Gamma})^\Gamma$, which classifies $\Gamma$-smoothable $\Gamma$-marked curves. The substack $(\ocM_{1,\Gamma})^\Gamma$ provides an initial approximation of $\cX(N)$, though it is still quite larger. Let us first see what it gives in the classical setting.

\begin{lem}\label{fixedlocus}
The modular curve $X(N)$ is a connected component $C$ of the fixed locus $(\ocM_{1,\Gamma})^\Gamma\otimes\ZZ[1/N]$, and the only connected component which satisfies the following two conditions:
\begin{itemize}
\item[(i)] Semistability: $C$ lies in the semistable locus $\ocM^\ss_{1,\Gamma}$ which classifies pairs $(E_S,\Gamma_S)$ with a semistable $E_S$.
\item[(ii)] Picard invariance: the action of $\Gamma$ on $E_C$ induces the trivial action on $\Pic^0_{E_C/C}$.
\end{itemize}
\end{lem}
\begin{proof}
Since $\ocM_1$ is regular, $\ocM_{1,\Gamma}\to\ocM_1$ is smooth and the action is tame once $N$ is inverted, $(\ocM_{1,\Gamma}\otimes\ZZ[1/N])^\Gamma$ is regular (e.g. see \cite[Proposition~4.2]{ExpVI}). The semistability condition is open by \cite{DR}. Automorphisms of generalized elliptic curves were described in \cite{DR}, and it follows that inside the semistable locus the Picard invariance is equivalent to the condition that the embedding $\Gamma\into\Aut_C(E_C)$ lands in the subgroup $\brE_C$, which is of finite index. Here and in the sequel the crown $\brE$ is as defined in \S\ref{crownsec}, so for a semistable genus-1 $E$ one simple has that $\brE=E^\sm$. This is a discrete condition, and hence, the Picard invariance specifies a closed and open subscheme $Z\subseteq\ocM^\ss_{1,\Gamma}$.

It now suffices to prove that the embedding of sets $X(N)\subseteq Z$ is an equality. If $x\in Z$, then $E_x$ is semistable, and hence any automorphism which preserves $\Pic^0_{E_x}$ is a translation by an element of the group $\brE_x$. Therefore, the $\Gamma$-action is through an embedding $\Gamma\subset \brE_x$ and it preserves the set of marked points $\Gamma_x$. This implies that $\Gamma_x=\brE_x[N]$, that is, $x\in X(N)$.
\end{proof}

We will not need the following complement.

\begin{rem}\label{nopicrem}
The Picard invariance condition can be removed when $N>3$ because in this case any embedding $\Gamma\subset\Aut(E_x)$ automatically lands in $\brE_x$, as can be checked by an explicit inspection. For example, if 6 is invertible, then $\Aut(E_x)/\brE_x$ is cyclic of order 2 if $j\notin\{0,1728\}$ and is of order 6 or 4 otherwise.
\end{rem}

Here are examples of other components in the fixed locus. They are not exhaustive but illustrate main ideas.

\begin{exam}\label{compexam}
(i) For $N=2$ any generalized elliptic curve $E$ over an algebraically closed field possesses an action of $\Gamma$ such that $\gamma_1$ induces the negation on $\Pic^0$. For example, one can take the action generated by the negation in $\brE$ and the translation by $P\in\brE[2]$ (this does not work if $E$ is supersingular in characteristic 2, but there are other actions then). Thus, there are other connected components in $(\ocM^\ss_{1,\Gamma})^\Gamma$ and there even exists a three-dimensional component coming form the $\Gamma$-marked curves of the form $(E,(Q,-Q,P+Q,P-Q))$, where $(E,O)$ is an elliptic curve and $Q$ varies (and the reflection fixes $O$ and $P$). For $N=3$ an analogous example exists only for the curve with $j=0$ and complex multiplication by $\ZZ[\xi_3]$.

(ii) Let $E$ be a genus-1 curve over $\overline{\QQ}$ whose components are an elliptic curve
$E_0$ and $\PP^1$'s $Z_0\.Z_{N-1}$ intersecting $E_0$ at the nodes $P_0\.P_{N-1}$ given by $P_i=iP_1$, where $P_1$ is of order $N$. Choose a coordinate $t_i$ on each $Z_i$ so that the node $P_i$ is at infinity and let $\gamma_{ij}$ be given by $t_i=\xi_N^j$ for $0\le i,j<N$. The action of $\Gamma$ on $E$ is generated by two elements: $\gamma_1$ translates $E_0$ by $P_1$ and permutes $Z_i$ accordingly, $\gamma_2$ acts by multiplication by $\xi_N$ on each $Z_i$ and fixes $E_0$. In the sequel, this type of curves is called Igusa case, see figure (\ref{types2}).
\end{exam}

\begin{rem}
(i) A pathology analogous to the first example also shows up in the definition of generalized elliptic curves, and it is the source of the condition in \cite{DR} that translations act by rotations on the $N$-gon (no inversions). We used the Picard invariance to get rid of such cases, and this will also be useful in the non-semistable case.

(ii) The second example does satisfy Picard invariance and it is ruled out in \cite{DR} by the semistability condition. When one has to work with non-semistable curves, as we will have to, one should use some other natural conditions to eliminate such cases. For example, either of the two conditions we will work with -- affinity or compatibility with Drinfeld level structure on the core, eliminates all not semistable components over $\ZZ[1/N]$.

(iii) When $N$ is not invertible the situation becomes substantially more complicated in various aspects. The fixed locus is very far from being regular, its ordinary components lie in the closure of a few different components over $\ZZ[1/N]$, while the supersingular ones are generically non-reduced. In particular, set theoretic conditions do not suffice to describe $\cX(N)$ inside the fixed locus. In this case only the combination of the two conditions will suffice to cut out $\cX(N)$.
\end{rem}

\section{$\Gamma$-marked genus-1 curves for $N=p$}\label{sec3}
Now let us restrict the generality to the main test case of this paper -- $N=p$ and $\Gamma=(\ZZ/p\ZZ)^2$. In this section we will study the moduli space $(\ocM_{1,\Gamma})^\Gamma$ of $\Gamma$-marked genus-1 curves as well as more basic moduli spaces and fixed loci. We will use the terminology of \S\ref{g1sec} (and the appendix).

\subsection{Geometric points}

\subsubsection{Canonical subgroups}
Let $\os=\Spec(k)$ be a geometric point of $(\ocM_{1,\Gamma})^\Gamma$ of characteristic $p$ and let $(E,\Gamma_k)$ be the corresponding $\Gamma$-marked $k$-curve. We also assume that the Picard invariance condition is satisfied  (which is only essential for $p=2,3$, see Remark~\ref{nopicrem}). The curve $E$ is not semistable as otherwise $\Gamma=\brE[p](k)$ would be at most of cardinality $p$. Thus, $E$ contains leaves (unstable components) and they all are stabilized by $\Gamma$. It follows that $\Gamma$ acts transitively on the set of leaves $\{Z_i\}$ and $\Gamma\subset\brE$. Let $Z_0$ be the leaf containing $0$ and $\Gamma_0=\Gamma\cap Z_0$. This gives a filtration $0\subsetneq \Gamma_0\subseteq\Gamma$. If $\Gamma_0=\ZZ/p\ZZ$ we call it the {\em canonical subgroup}, otherwise we say that $(E,\Gamma_k)$ is {\em deeply supersingular}.

\subsubsection{Classification in characteristic $p$}\label{classec}
Since $\Gamma_0$ acts on $Z_0$ and fixes the node, it acts through the group $\Aut(\PP^1_k,\infty)=\GG(k)$, where $\GG=\GG_a\rtimes\GG_m$. Since $\GG_m(k)$ contains no $p$-torsion, the action is through $\GG_a$ and there exists a choice of the coordinate $t$ on $Z_0$ so that $t=\infty$ at the node, $t(O_\Gamma)=0$ and $t(\gamma_1)=1$.

In the deeply supersingular case we can take $\gamma=\gamma_1$, making the choice of the coordinate $t$ unique, and then $t(\gamma_2)=\lam\in k\setminus\FF_p$. In this case there is just one $\Gamma$-stable configuration, and $E$ consists of two components: an irreducible core $E_0$ (smooth or a $1$-gon) and the unique leaf $Z=Z_0$, see figure (\ref{types1}) below.

If $\Gamma_0$ is cyclic, then there are $p+1$ choices for $\Gamma_0\subset\Gamma$, and we assume for concreteness that $\gamma_1\in\Gamma_0$. Each leaf is a $\PP^1_k$ acted on by $\FF_p$ and there are $p$ leaves acted on transitively by $\Gamma/\Gamma_0$. Consider the next layer which consists of the components that intersect the leaves. Of course, $\Gamma$ acts transitively on this layer and there are two cases. If this layer is not in the semistable core of $E$, then it consists of the single component $Z=\PP^1_k$ and we can choose the coordinate $t'$ so that $\infty$ is the node in the semistable core $E_0$ and $\FF_p$ are the nodes where the leaves are attached. The group $\Gamma$ acts on $Z$ with the stabilizer $\Gamma_0$ and translates the finite nodes. In this case, $E_0$ is irreducible, and we say that $E$ is {\em canonical supersingular}.

Finally, it can happen that the $p$ nodes of the leaves lie in $E_0$. Since $\Gamma/\Gamma_0$ acts on this set of nodes transitively it follows that it is precisely $E_0[p](k)$ and hence either $E_0$ is ordinary elliptic or a $p$-gon and its torsion is rigidified by $\Gamma/\Gamma_0$. We call these cases {\em ordinary} and {\em Tate} accordingly. We summarize the cases as follows.

\begin{itemize}
\item[(i)] {\em Deep supersingular.} The semistable core is irreducible and there is one leaf with $\FF_p\oplus\lam\FF_p$ marked points.
\item[(ii)] {\em Ordinary.} The core $E_0$ is an ordinary elliptic curve. There are $p$ leaves with nodes at $E_0[p](k)$, and each leaf contains $\FF_p$ marked points.
\item[(iii)] {\em Canonical supersingular.} The core is irreducible and there are one intermediate component $Z$ and $p$ leaves intersecting it at the points of $\FF_p$. The marked points on the leaves are also at the points of $\FF_p$. This configuration can generalize to the deep supersingular component by smoothing the $p$ nodes between $Z$ and the leaves. If $E_0$ is supersingular, it can only generalize to an ordinary component by smoothing the node between $E_0$ and $Z$.
\item[(iv)] {\em Tate's ordinary configuration.} The core is a $p$-gon, the next layer consists of $p$ leaves with $\FF_p$ marked points. This case generalizes to the ordinary case by smoothing the nodes of the $p$-gon.
\end{itemize}

Here is an illustration for $p=3$ in terms of the dual graphs. Nodes and their smoothings under generizations correspond to edges and their contractions, which are indicated by arrows.\m

% Define universal TikZ styles
\tikzset{
    node_compressed/.style={
        circle,
        draw,
        thick,
        minimum size=0.7cm,
        inner sep=0pt,
        fill=white,
        font=\normalsize\bfseries
    },
    edge_thick/.style={
        thick
    },
    leg/.style={
        thick,
        -
    },
    type_label/.style={
        font=\Large\bfseries,
        text depth=0.5ex
    },
    contraction_arrow/.style={
        stealth-, % Reversed so the arrow points to the first coordinate
        line width=1.5pt,
        blue!70!black
    },
    contraction_text/.style={
        font=\normalsize\bfseries,
        blue!70!black,
        align=center
    },
    % Edge label style for neat placement
    edge_label/.style={
        font=\small\bfseries,
        inner sep=2pt
    }
}

\begin{equation}
% scale=0.77 shrinks the whole diagram by roughly 1.3 (since 1/1.3 ≈ 0.77)
\begin{tikzpicture}[scale=0.77, transform shape]

    % Vertical distance and horizontal leaf spread
    \pgfmathsetmacro{\yStep}{1.8}
    \pgfmathsetmacro{\xSpread}{1.1}

    % Explicit X-coordinates for each of the 4 graphs to finely tune spacing
    \pgfmathsetmacro{\xOne}{0}       % Graph 1: Deep
    \pgfmathsetmacro{\xTwo}{3.3}     % Graph 2: Canonical (3.3cm from Deep)
    \pgfmathsetmacro{\xThree}{7.6}   % Graph 3: Ordinary (4.3cm from Canonical)
    \pgfmathsetmacro{\xFour}{12.4}   % Graph 4: Tate (increased by 0.5cm)

    % =========================================================================
    % Type 1: Deep Supersingular (Root E0 -> single Leaf Z, 9 legs)
    % =========================================================================
    \begin{scope}[xshift=\xOne cm]
        \node[node_compressed] (E0_1) at (0, 2*\yStep) {$E_0$};
        \node[node_compressed] (Z_1) at (0, 1*\yStep) {$Z$};

        \draw[edge_thick] (E0_1) -- (Z_1) node[midway, left, edge_label] {$l_0$};

        % 9 legs spread around 270 deg (bottom)
        \foreach \i in {1,...,9}
        {
            \pgfmathsetmacro{\angle}{270 + (\i-5)*10}
            \draw[leg] (Z_1) -- +(\angle:0.7);
        }

        \node[type_label] at (0, -1.8) {Deep};
    \end{scope}

    % =========================================================================
    % Type 2: Canonical Supersingular (Root E0 -> Branch Z -> Leaves Z0...2)
    % =========================================================================
    \begin{scope}[xshift=\xTwo cm]
        \node[node_compressed] (E0_2) at (0, 2*\yStep) {$E_0$};
        \node[node_compressed] (Z_branch) at (0, 1*\yStep) {$Z$};

        \node[node_compressed] (Z0_2) at (-\xSpread, 0) {$Z_0$};
        \node[node_compressed] (Z1_2) at (0, 0)          {$Z_1$};
        \node[node_compressed] (Z2_2) at (\xSpread, 0)  {$Z_2$};

        % Edge l_0
        \draw[edge_thick] (E0_2) -- (Z_branch) node[midway, right, edge_label] {$l_0$};

        % Edges l_1
        \draw[edge_thick] (Z_branch) -- (Z0_2) node[midway, left, edge_label, xshift=-2pt] {$l_1$};
        \draw[edge_thick] (Z_branch) -- (Z1_2) node[pos=0.6, right, edge_label] {$l_1$};
        \draw[edge_thick] (Z_branch) -- (Z2_2) node[midway, right, edge_label, xshift=2pt] {$l_1$};

        % 3 legs per leaf (9 total)
        \foreach \i in {1,2,3}
        {
            \pgfmathsetmacro{\angle}{270 + (\i-2)*25}
            \draw[leg] (Z0_2) -- +(\angle:0.7);
            \draw[leg] (Z1_2) -- +(\angle:0.7);
            \draw[leg] (Z2_2) -- +(\angle:0.7);
        }

        \node[type_label] at (0, -1.8) {Canonical};
    \end{scope}

    % =========================================================================
    % Type 3: Ordinary (Root E_0 -> Leaves Z0...2, 3 legs per leaf)
    % =========================================================================
    \begin{scope}[xshift=\xThree cm]
        \node[node_compressed, fill=green!5] (E_3) at (0, 1*\yStep) {$E_0$};
        \node[node_compressed] (Z0_3) at (-\xSpread, 0) {$Z_0$};
        \node[node_compressed] (Z1_3) at (0, 0)          {$Z_1$};
        \node[node_compressed] (Z2_3) at (\xSpread, 0)  {$Z_2$};

        % Edges l_1
        \draw[edge_thick] (E_3) -- (Z0_3) node[midway, left, edge_label, xshift=-2pt] {$l_1$};
        \draw[edge_thick] (E_3) -- (Z1_3) node[pos=0.6, right, edge_label] {$l_1$};
        \draw[edge_thick] (E_3) -- (Z2_3) node[midway, right, edge_label, xshift=2pt] {$l_1$};

        % 3 legs per leaf (9 total)
        \foreach \i in {1,2,3}
        {
            \pgfmathsetmacro{\angle}{270 + (\i-2)*25}
            \draw[leg] (Z0_3) -- +(\angle:0.7);
            \draw[leg] (Z1_3) -- +(\angle:0.7);
            \draw[leg] (Z2_3) -- +(\angle:0.7);
        }

        \node[type_label] at (0, -1.8) {Ordinary};
    \end{scope}

    % =========================================================================
    % Type 4: Tate / Bad Reduction (Cycle of length 3)
    % =========================================================================
    \begin{scope}[xshift=\xFour cm]

        % Nested scope to handle the vertical shift cleanly
        \begin{scope}[yshift=1.05*\yStep cm]
            \pgfmathsetmacro{\rCyc}{0.85}        % Radius for the triangle vertices
            \pgfmathsetmacro{\rLf}{2.0}          % Radius for the leaves

            % Triangle Cycle (Vertices: E_0, E_1, E_2)
            \node[node_compressed] (C1) at (90:\rCyc) {$E_2$};
            \node[node_compressed] (C2) at (210:\rCyc) {$E_0$};
            \node[node_compressed] (C3) at (330:\rCyc) {$E_1$};

            % Edges of the triangle, labelled 'l'
            \draw[edge_thick] (C1) -- (C2) node[midway, above left, edge_label] {$l$};
            \draw[edge_thick] (C2) -- (C3) node[midway, below, edge_label, yshift=-2pt] {$l$};
            \draw[edge_thick] (C3) -- (C1) node[midway, above right, edge_label] {$l$};

            % Leaves (Z_0, Z_1, Z_2)
            \node[node_compressed] (L1) at (90:\rLf) {$Z_2$};
            \node[node_compressed] (L2) at (210:\rLf) {$Z_0$};
            \node[node_compressed] (L3) at (330:\rLf) {$Z_1$};

            % Connecting edges
            \draw[edge_thick] (C1) -- (L1) node[midway, right, edge_label, xshift=2pt] {$l_1$};
            \draw[edge_thick] (C2) -- (L2) node[midway, right, edge_label, xshift=-1pt, yshift=-6pt] {$l_1$};
            \draw[edge_thick] (C3) -- (L3) node[midway, right, edge_label, xshift=-2pt, yshift=6pt] {$l_1$};

            % Legs (3 per leaf, projecting outwards)
            \foreach \i in {1,2,3} {
                \pgfmathsetmacro{\angle}{90 + (\i-2)*30}
                \draw[leg] (L1) -- +(\angle:0.7);
            }
            \foreach \i in {1,2,3} {
                \pgfmathsetmacro{\angle}{210 + (\i-2)*30}
                \draw[leg] (L2) -- +(\angle:0.7);
            }
            \foreach \i in {1,2,3} {
                \pgfmathsetmacro{\angle}{330 + (\i-2)*30}
                \draw[leg] (L3) -- +(\angle:0.7);
            }
        \end{scope}

        % Label correctly aligned with the other three graphs
        \node[type_label] at (0, -1.8) {Tate};
    \end{scope}

    % =========================================================================
    % CONTRACTION ARROWS
    % =========================================================================
    % Note: The stealth- style draws the arrow tip at the FIRST coordinate.

    % Arrow from Type 2 to Type 1 (Contracting bottom edges l_1)
    % Points leftwards from Canonical to Deep
    \draw[contraction_arrow]
        (\xOne + 1.4, 0.8*\yStep) -- (\xTwo - 1.1, 0.65*\yStep)
        node[midway, above, contraction_text] {$l_1$};

    % Arrow from Type 2 to Type 3 (Contracting upper edge l_0)
    % Points rightwards from Canonical to Ordinary
    \draw[contraction_arrow]
        (\xThree - 1.3, 1.3*\yStep) -- (\xTwo + 1.3, 1.45*\yStep)
        node[midway, above, contraction_text] {$l_0$};

    % Arrow from Type 3 to Type 4 (Contracting cycle edges l)
    % Points rightwards from Ordinary to Tate
    \draw[contraction_arrow]
        (\xThree + 1.6, 1.05*\yStep) -- (\xFour - 2.3, 1.05*\yStep)
        node[midway, above, contraction_text] {$l$};
\end{tikzpicture}
\label{types1}
\end{equation}

\subsubsection{Classification when $p$ is invertible}\label{pinvsec}
When $p\neq\cha(k)$ the same combinatorial arguments also lead to a full list of $\Gamma$-marked configurations. The only differences are that the $p$-torsion is always maximal: $E[p]\toisom\Gamma$ and $\GG_m[p]=\mu_p$, the additive group has trivial torsion, so the action on rational components is through $\mu_p\subset\GG_m$ and there are no faithful $\Gamma$-actions on rational components. Comparing to the case of $\cha(k)=p$ this eliminates the supersingular configuration, adds the usual semistable configurations of elliptic curves and Tate $p$-gons, and forbids the specialization from ordinary to canonical. Here is a brief summary of configurations:

\begin{itemize}
\item[(o)] {\em Smooth.} A usual $\Gamma$-leveled elliptic curve.
\item[(i)] {\em Tate's semistable configuration.} A usual $\Gamma$-leveled Tate's $p$-gon. It can generalize to a smooth curve.
\item[(ii)] {\em Igusa/ordinary configuration.} The core is an elliptic curve. There are $p$ leaves with the nodes given by a subgroup $\Gamma_0\subset E_0[p]$, and each leaf contains $\mu_p$ marked points.
\item[(iii)] {\em Tate's ordinary configuration.} The core is a $p$-gon, the next layer consists of $p$ leaves with $\mu_p$ marked points. This case generalizes to the ordinary case by smoothing the nodes of the $p$-gon.
\item[(iv)] {\em Canonical supersingular.} The core is irreducible and there are one intermediate component $Z$ and $p$ leaves hanging at $\mu_p$. Each leaf contains $\mu_p$ marked points.
\end{itemize}

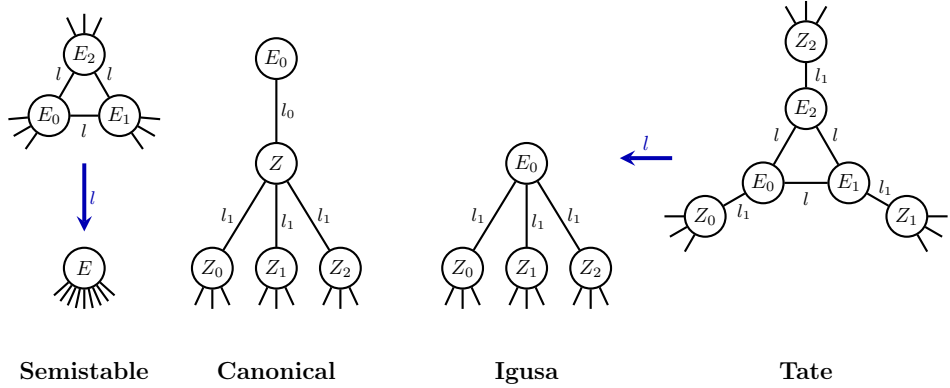
\begin{equation}
% scale=0.77 shrinks the whole diagram by roughly 1.3 (since 1/1.3 ≈ 0.77)
\begin{tikzpicture}[scale=0.77, transform shape]

    % Vertical distance and horizontal leaf spread
    \pgfmathsetmacro{\yStep}{1.8}
    \pgfmathsetmacro{\xSpread}{1.1}

    % Explicit X-coordinates for each of the 4 graphs to finely tune spacing
    \pgfmathsetmacro{\xOne}{0}       % Graph 1: Semistable
    \pgfmathsetmacro{\xTwo}{3.3}     % Graph 2: Canonical (3.3cm from Semistable)
    \pgfmathsetmacro{\xThree}{7.6}   % Graph 3: Ordinary (4.3cm from Canonical)
    \pgfmathsetmacro{\xFour}{12.4}   % Graph 4: Tate (increased by 0.5cm)

    % =========================================================================
    % Type 1: Semistable (Triangle with 3 legs/node -> single node Z, 9 legs)
    % =========================================================================
    \begin{scope}[xshift=\xOne cm]
        % Top component: Triangle Cycle (Vertices: E_0, E_1, E_2)
        \begin{scope}[yshift=1.65*\yStep cm]
            \pgfmathsetmacro{\rCyc}{0.7}
            \node[node_compressed] (C1_1) at (90:\rCyc) {$E_2$};
            \node[node_compressed] (C2_1) at (210:\rCyc) {$E_0$};
            \node[node_compressed] (C3_1) at (330:\rCyc) {$E_1$};

            % Edges of the triangle, labelled 'l'
            \draw[edge_thick] (C1_1) -- (C2_1) node[midway, above left, edge_label] {$l$};
            \draw[edge_thick] (C2_1) -- (C3_1) node[midway, below, edge_label, yshift=-2pt] {$l$};
            \draw[edge_thick] (C3_1) -- (C1_1) node[midway, above right, edge_label] {$l$};

            % 3 legs per node projecting outwards
            \foreach \i in {1,2,3} {
                \pgfmathsetmacro{\angle}{90 + (\i-2)*25}
                \draw[leg] (C1_1) -- +(\angle:0.7);
            }
            \foreach \i in {1,2,3} {
                \pgfmathsetmacro{\angle}{210 + (\i-2)*25}
                \draw[leg] (C2_1) -- +(\angle:0.7);
            }
            \foreach \i in {1,2,3} {
                \pgfmathsetmacro{\angle}{330 + (\i-2)*25}
                \draw[leg] (C3_1) -- +(\angle:0.7);
            }
        \end{scope}

        % Bottom component: Single node with 9 legs
        \node[node_compressed] (Z_1) at (0, 0) {$E$};

        % 9 legs spread around 270 deg (bottom)
        \foreach \i in {1,...,9}
        {
            \pgfmathsetmacro{\angle}{270 + (\i-5)*12}
            \draw[leg] (Z_1) -- +(\angle:0.7);
        }

        % Contraction arrow pointing from the triangle down to the single node with label 'l'
        \draw[contraction_arrow] (0, 0.35*\yStep) -- (0, 1.*\yStep) node[midway, right, edge_label] {$l$};

        \node[type_label] at (0, -1.8) {Semistable};
    \end{scope}

    % =========================================================================
    % Type 2: Canonical Supersingular (Root E0 -> Branch Z -> Leaves Z0...2)
    % =========================================================================
    \begin{scope}[xshift=\xTwo cm]
        \node[node_compressed] (E0_2) at (0, 2*\yStep) {$E_0$};
        \node[node_compressed] (Z_branch) at (0, 1*\yStep) {$Z$};

        \node[node_compressed] (Z0_2) at (-\xSpread, 0) {$Z_0$};
        \node[node_compressed] (Z1_2) at (0, 0)          {$Z_1$};
        \node[node_compressed] (Z2_2) at (\xSpread, 0)  {$Z_2$};

        % Edge l_0
        \draw[edge_thick] (E0_2) -- (Z_branch) node[midway, right, edge_label] {$l_0$};

        % Edges l_1
        \draw[edge_thick] (Z_branch) -- (Z0_2) node[midway, left, edge_label, xshift=-2pt] {$l_1$};
        \draw[edge_thick] (Z_branch) -- (Z1_2) node[pos=0.6, right, edge_label] {$l_1$};
        \draw[edge_thick] (Z_branch) -- (Z2_2) node[midway, right, edge_label, xshift=2pt] {$l_1$};

        % 3 legs per leaf (9 total)
        \foreach \i in {1,2,3}
        {
            \pgfmathsetmacro{\angle}{270 + (\i-2)*25}
            \draw[leg] (Z0_2) -- +(\angle:0.7);
            \draw[leg] (Z1_2) -- +(\angle:0.7);
            \draw[leg] (Z2_2) -- +(\angle:0.7);
        }

        \node[type_label] at (0, -1.8) {Canonical};
    \end{scope}

    % =========================================================================
    % Type 3: Ordinary (Root E_0 -> Leaves Z0...2, 3 legs per leaf)
    % =========================================================================
    \begin{scope}[xshift=\xThree cm]
        \node[node_compressed] (E_3) at (0, 1*\yStep) {$E_0$};
        \node[node_compressed] (Z0_3) at (-\xSpread, 0) {$Z_0$};
        \node[node_compressed] (Z1_3) at (0, 0)          {$Z_1$};
        \node[node_compressed] (Z2_3) at (\xSpread, 0)  {$Z_2$};

        % Edges l_1
        \draw[edge_thick] (E_3) -- (Z0_3) node[midway, left, edge_label, xshift=-2pt] {$l_1$};
        \draw[edge_thick] (E_3) -- (Z1_3) node[pos=0.6, right, edge_label] {$l_1$};
        \draw[edge_thick] (E_3) -- (Z2_3) node[midway, right, edge_label, xshift=2pt] {$l_1$};

        % 3 legs per leaf (9 total)
        \foreach \i in {1,2,3}
        {
            \pgfmathsetmacro{\angle}{270 + (\i-2)*25}
            \draw[leg] (Z0_3) -- +(\angle:0.7);
            \draw[leg] (Z1_3) -- +(\angle:0.7);
            \draw[leg] (Z2_3) -- +(\angle:0.7);
        }

        \node[type_label] at (0, -1.8) {Igusa};
    \end{scope}

    % =========================================================================
    % Type 4: Tate / Bad Reduction (Cycle of length 3)
    % =========================================================================
    \begin{scope}[xshift=\xFour cm]

        % Nested scope to handle the vertical shift cleanly
        \begin{scope}[yshift=1.05*\yStep cm]
            \pgfmathsetmacro{\rCyc}{0.85}        % Radius for the triangle vertices
            \pgfmathsetmacro{\rLf}{2.0}          % Radius for the leaves

            % Triangle Cycle (Vertices: E_0, E_1, E_2)
            \node[node_compressed] (C1) at (90:\rCyc) {$E_2$};
            \node[node_compressed] (C2) at (210:\rCyc) {$E_0$};
            \node[node_compressed] (C3) at (330:\rCyc) {$E_1$};

            % Edges of the triangle, labelled 'l'
            \draw[edge_thick] (C1) -- (C2) node[midway, above left, edge_label] {$l$};
            \draw[edge_thick] (C2) -- (C3) node[midway, below, edge_label, yshift=-2pt] {$l$};
            \draw[edge_thick] (C3) -- (C1) node[midway, above right, edge_label] {$l$};

            % Leaves (Z_0, Z_1, Z_2)
            \node[node_compressed] (L1) at (90:\rLf) {$Z_2$};
            \node[node_compressed] (L2) at (210:\rLf) {$Z_0$};
            \node[node_compressed] (L3) at (330:\rLf) {$Z_1$};

            % Connecting edges
            \draw[edge_thick] (C1) -- (L1) node[midway, right, edge_label, xshift=2pt] {$l_1$};
            \draw[edge_thick] (C2) -- (L2) node[midway, right, edge_label, xshift=-1pt, yshift=-6pt] {$l_1$};
            \draw[edge_thick] (C3) -- (L3) node[midway, right, edge_label, xshift=-2pt, yshift=6pt] {$l_1$};

            % Legs (3 per leaf, projecting outwards)
            \foreach \i in {1,2,3} {
                \pgfmathsetmacro{\angle}{90 + (\i-2)*30}
                \draw[leg] (L1) -- +(\angle:0.7);
            }
            \foreach \i in {1,2,3} {
                \pgfmathsetmacro{\angle}{210 + (\i-2)*30}
                \draw[leg] (L2) -- +(\angle:0.7);
            }
            \foreach \i in {1,2,3} {
                \pgfmathsetmacro{\angle}{330 + (\i-2)*30}
                \draw[leg] (L3) -- +(\angle:0.7);
            }
        \end{scope}

        % Label correctly aligned with the other three graphs
        \node[type_label] at (0, -1.8) {Tate};
    \end{scope}

    % =========================================================================
    % CONTRACTION ARROWS
    % =========================================================================
    % Note: The stealth- style draws the arrow tip at the FIRST coordinate.

    % Arrow from Type 3 to Type 4 (Contracting cycle edges l)
    % Points rightwards from Ordinary to Tate
    \draw[contraction_arrow]
        (\xThree + 1.6, 1.05*\yStep) -- (\xFour - 2.3, 1.05*\yStep)
        node[midway, above, contraction_text] {$l$};

\end{tikzpicture}
\label{types2}
\end{equation}

\subsubsection{$\Gamma$-smoothability}
Using the above classification we can also describe the $\Gamma$-smoothability condition. When $p$ is invertible the situation is classical and was described in \cite{DR} -- the only smoothing is from the semistable Tate configuration to an elliptic curve, and the other components of $(\ocM_{1,\Gamma})^\Gamma$ are non-smoothable.

\begin{lem}\label{Gammasmoothlem}
Let $(E,\Gamma_k)$ be a $\Gamma$-marked genus-1 curve over an algebraically closed field of characteristic $p$. Then $(E,\Gamma_k)$ is always $\Gamma$-smoothable in the ordinary and Tate cases, and it is smoothable in the supersingular cases if and only if the core $E_0$ is supersingular.
\end{lem}
\begin{proof}
In this case the $\Gamma$-smoothability condition is equivalent to existence of a valuation ring $R$ with $k=R/m_R$ and a $\Gamma$-marked $R$-curve $(E_R,\Gamma_R)$ whose closed fiber is $(E,\Gamma_k)$ and generic fiber $(E_\eta,\Gamma_\eta)$ is $\Gamma$-leveled elliptic. If such an $E_R$ exists, then $\Gamma_\eta=E_\eta[p]$ specializes to $E_0[p]$, and hence the image of $\Gamma_k$ under the retraction onto the core is $E_0[p](k)$. This implies that the conditions are necessary for $\Gamma$-smoothability.

Conversely, we should construct a lift for each type as in the lemma. In the ordinary, Tate and canonical supersingular cases everything is controlled by combinatorics and the isomorphism type of the core, so we can take any ordinary good reduction $E_\eta$ which reduces to $E_0$, any bad reduction $E_\eta$ or any $E_\eta$ with supersingular reduction and canonical subgroup, and this already provides a $\Gamma$-smoothing.

In the deep supersingular case there is a one-dimensional component with coordinate $\lam$, where $\Gamma=\FF_p\oplus\lam\FF_p$ in the leaf. We should prove that the whole component is $\Gamma$-smoothable. So, it suffices to prove that the set of $\Gamma$-smoothable values of $\lam$ is infinite. One way to check this is as follows. Consider the Katz-Mazur model $\tilcX=\tilcX(p)$ and let $\tilx$ be the supersingular point corresponding to $j(E_0)$. Blowing up $\tilx$ produces a divisor $D$ with generic point $x$, and $R=\cO_x$ is a mixed characteristic DVR. Since $D$ intersects $p+1$ Igusa components separated by this blowup, $D$ contains supersingular points with different canonical subgroups. It follows that $(E_x,\Gamma_x)$ has deep supersingular reduction and the corresponding value of $\lam_x\in k(x)$ is not constant on $D$, and hence is transcendental over $\FF_q$.
\end{proof}

Recall that in the Katz-Mazur model $\tilcX(p)$ the ordinary locus consists of $p+1$ Igusa curves intersecting transversally at each supersingular point $z_j$, and the set of supersingular points is the vanishing locus $V(S_p)$ of an appropriate separable polynomial $S_p(j)$ of approximate degree $\frac{p}{12}$ (the precise weighted counting with automorphisms yields $\frac{p-1}{24}$). Hence blowing up this locus inserts new components $Z_j=\PP^1_{k(x_j)}$ and separates Igusa curves. Note that $\GL_2(\FF_p)$ acts on $X(p)$ through automorphisms of $\Gamma$ (and it is the Galois group of $X(p)$ over $\PP^1_j\otimes\ZZ[1/p]$), hence it acts transitively on the set of Igusa curves and for an appropriate coordinate $\lam_j$ on $Z_j$ the intersection with the Igusa curves occurs at $\FF_p$-points. Combining this with the previous lemma we obtain the following comparison result.

\begin{cor}\label{comparecor}
Let $\Bl_{S_p}(\tilcX(p))$ denote the blowing up of the Katz-Mazur model $\tilcX(p)$ at the set of supersingular points. Then the sets of geometric points of $\cX(p)$ and $\Bl_{S_p}(\tilcX(p))$ are naturally bijective.
\end{cor}

\subsubsection{Ramified $\Gamma$-smoothings}
We have already used that smoothability of a $\Gamma$-marked $k$-curve is equivalent to existence of a smoothing $(E_R,\Gamma_R)$ over a mixed characteristic valuation ring $R$. We say that such a smoothing is {\em ramified} at a node $P\in E$ if the modulus $\pi_P\in R$ of $P$ in $E_R$ strictly divides $p$ (i.e. $(\pi_R)\subsetneq(p)$). In fact, if either $p>2$ or the configuration of $E_k$ is supersingular, then any smoothing of $E_k$ is ramified at any node. Furthermore, we have the following very concrete rule the smoothing parameters of the nodes obey.

\begin{lem}\label{ramlem}
Assume that $R$ is a mixed characteristic DVR of residual characteristic $p$ and $(E_R,\Gamma_R)$ a generically smooth $\Gamma$-marked $R$-curve. In particular, $E_k$ is a $\Gamma$-marked curve of one of four types from \S\ref{classec} diagram (\ref{types1}). Let $\pi_i\in R$ be the modulus of the nodes of type $l_i$ in the diagram, where we set $\pi_i=1$ if a node of type $l_i$ does not exist in $E_k$. Then $(p)=(\pi_0^{p^2-1}\pi_1^{p-1})$.
\end{lem}
\begin{proof}
Consider the formal group $\hatcE$ of $\cE=E_R$. One can view it either as a formal disc $\Spf(R\llbracket t\rrbracket)$ with a group law or as its generic fiber, which is an open unit non-archimedean disc with a coordinate $t$ and an appropriate group law. Multiplication by $p$ in the group law is given by a series $f_p(t)=pt+\dots$, which acts by $p$ on the tangent space and hence equals $pt$ modulo $t^2$. Also, the group is of height $h=1$ in the ordinary and Tate cases and $h=2$ in the supersingular case, and hence $f_p(t)$ is $t^{p^h}$ modulo $p$. Extending $R$ if necessary, we can assume that all $p^h$ points of $\hatcE[p]$ are in $R$ and hence $f_p(t)=u(t)\prod_{P\in\hatcE[p](R)}(t-t(P))$, where $u$ equals 1 modulo $p$ and we use Weierstrass division to factor it out.

Consider for concreteness the case of a canonical supersingular configuration, when both $\pi_0$ and $\pi_1$ are not invertible. The other cases are dealt with similarly, and can be formally obtained as limits when the appropriate $\pi_i$ tends to 1. Set $\pi=\pi_0\pi_1$ and let $Z_0$ be the leaf containing 0 and $Z$ the rational component above it, see diagram (\ref{types1}). The natural coordinates on $Z$ and $Z_0$ are $t/\pi_0$ and $t/\pi$, respectively (and the local rings at the generic points are the valuation rings corresponding to the Gauss valuations for the coordinates $t/\pi_0$ and $t/\pi$). In particular, each of $p-1$ non-zero torsion points $P\in\cE[p](R)$ in the canonical subgroup satisfies $|t(P)|=|\pi|$ and each point $Q\in\cE[p](R)$ not lying in the canonical subgroup satisfies $|t(Q)|=|\pi_0|$ with respect to the $p$-adic valuation of $R$. Since $p=\prod_{P\in\cE[p](R)\setminus\{0\}}t(P)$ this implies that $|p|=|\pi_0^{p^2-1}\pi_1^{p-1}|$, as claimed.
\end{proof}

\subsubsection{The geometric locus}
For the sake of illustration let us expand a bit about the whole structure and draw a picture of (the underlying set of) $(\ocM_{1,\Gamma}\otimes\FF_p)^\Gamma$. The {\em supersingular locus} is just $\PP^1_j\times\PP^1_\lam$ with the deep locus given by $\lam\notin\PP^1(\FF_p)$ and $p+1$ components of the complement where the locus is canonical. As we saw, the $\Gamma$-smoothable locus in $\PP^1_j\times\PP^1_\lam$ corresponds to the supersingular values of $j$, that is, the roots of the separable polynomial $S_p(j)$ we mentioned earlier. So, it consists of approximately $\frac{p}{12}$ components $\PP^1_\lam$.

The {\em ordinary locus} consists of $p+1$ copies of the usual Igusa curve $\Ig_p$ parameterizing generalized elliptic curves over $\FF_p$ with a trivialization $E[p](k)\toisom\ZZ/p\ZZ$. The missing points of its projective compactification are precisely the $\Gamma$-smoothable points of the canonical supersingular locus.

Here is an illustration for $p=3$. The cylinder is the supersingular locus, its $\Gamma$-smoothable part is marked by blue and obtained when $S_p(j)=0$, it contains the canonical $\Gamma$-smoothable part marked by black points, and the canonical locus also lies on the Igusa curves, forming the red ordinary locus, which contains the green Tate locus with $j=\infty$. The smoothable locus is colored. The Igusa curves are of degree $p-1$ over $\PP^1_j$ -- \'etale over the ordinary locus and totally ramified over the supersingular values of $j$. Finally, the equation $S_p(j)=0$ has only one solution for $p\le 7$, but we draw two such sections to illustrate the fact that there are in general several supersingular curves.

\begin{equation}
\begin{tikzpicture}[scale=1]

    % --- j-AXIS ---
    \draw[->, dashed, gray!80] (0, -2.8) -- (0, 2.8) node[right, black] {$j$};

    % --- CYLINDER BACK ELEMENTS ---
    % Bottom ellipse of the cylinder (back half)
    \draw[gray, thick, dashed] (4, -2) arc (0:180:4cm and 0.5cm);

    % Middle ellipse for j=\infty (back half)
%    \draw[green!70!black, thick, dashed] (4, 0) arc (0:180:4cm and 0.5cm);
    \draw[gray, thick, dashed] (4, 0) arc (0:180:4cm and 0.5cm);

    % --- SUPERSINGULAR OVALS (BACK HALVES) ---
    % Back dashed arc for top oval (j = 0.8)
    \draw[blue, thick, dashed] (4, 0.8) arc (0:180:4cm and 0.5cm);
    % Back dashed arc for bottom oval (j = -0.8)
    \draw[blue, thick, dashed] (4, -0.8) arc (0:180:4cm and 0.5cm);

    % --- IGUSA CURVES (BACK) ---
    % The curves are now closed ovals in vertical planes for each angle.
    % Parametric radius R(s) = 4 + 1.2*cos(s), height z(s) = 0.8*sin(s)
    % Angles 18 and 108 are in the back.

    % Curve: angle = 108 (Back Left)
    \draw[thick, red, domain=0:360, variable=\s, smooth, samples=60]
        plot ({ (4 + 1.2*cos(\s)) * cos(108) },
              { 0.125 * (4 + 1.2*cos(\s)) * sin(108) + 0.8*sin(\s) });

    % Curve: angle = 18 (Back Right)
    \draw[thick, red, domain=0:360, variable=\s, smooth, samples=60]
        plot ({ (4 + 1.2*cos(\s)) * cos(18) },
              { 0.125 * (4 + 1.2*cos(\s)) * sin(18) + 0.8*sin(\s) });

    % --- SUPERSINGULAR OVALS (FRONT HALVES) ---
    % Front solid arc for top oval
    \draw[blue, thick] (-4, 0.8) arc (180:360:4cm and 0.5cm);
    \node[blue] at (-2.9, 0.8) {\large $S_p(j)=0$};

    % Front solid arc for bottom oval
    \draw[blue, thick] (-4, -0.8) arc (180:360:4cm and 0.5cm);
    \node[blue] at (-2.2, -0.8) {\large $S_p(j)=0$};

    % --- CYLINDER FRONT & TOP ---
    % Top ellipse of the cylinder (z = 2)
    \draw[gray, thick] (0, 2) ellipse (4cm and 0.5cm);

    % Middle ellipse for j=\infty (front half)
    \draw[gray, thick] (-4, 0) arc (180:360:4cm and 0.5cm);
    \node[green!70!black] at (0.1, 0) {\large $j=\infty$};

    % Bottom ellipse of the cylinder (front half)
    \draw[gray, thick] (-4, -2) arc (180:360:4cm and 0.5cm);

    % Vertical bounding lines of the cylinder (silhouette edges)
    \draw[gray, thick] (-4, -2) -- (-4, 2);
    \draw[gray, thick] (4, -2) -- (4, 2);

    \draw[gray, thick] (-3.8, -2.15) -- (-3.8, 1.85);
    \draw[gray, dashed, thick] (3.8, -1.85) -- (3.8, 2.17);
    \draw[gray, thick] (1.24, -2.48) -- (1.24, 1.52);
    \draw[gray, dashed, thick] (-1.23, -1.52) -- (-1.23, 2.48);

    % --- IGUSA CURVES (FRONT) ---
    % Angles 198 and 288 are in the front.

    % Curve: angle = 198 (Front Left)
    \draw[thick, red, domain=0:360, variable=\s, smooth, samples=60]
        plot ({ (4 + 1.2*cos(\s)) * cos(198) },
              { 0.125 * (4 + 1.2*cos(\s)) * sin(198) + 0.8*sin(\s) });

    % Curve: angle = 288 (Front Right)
    \draw[thick, red, domain=0:360, variable=\s, smooth, samples=60]
        plot ({ (4 + 1.2*cos(\s)) * cos(288) },
              { 0.125 * (4 + 1.2*cos(\s)) * sin(288) + 0.8*sin(\s) });

    % --- LABELS FOR Ig_p ---
    % Placed right outside the maximum radius (R=5.2) at the z=0 equator
%   \node[red] at ({5.8*cos(18)}, {0.725*sin(18)}) {$Ig^{0}_p$};
%   \node[red] at ({6.*cos(108)}, {0.725*sin(108)}) {$Ig^1_p$};
%   \node[red] at ({5.8*cos(198)}, {0.725*sin(198)}) {$Ig^2_p$};
%   \node[red] at ({6.*cos(288)}, {0.725*sin(288)}) {$Ig^\infty_p$};
    \node[red] at (2, -.95) {$Ig^1_p$};
    \node[red] at (4.5, 0) {$Ig^2_p$};
    \node[red] at (-.6, 1) {$Ig^\infty_p$};
    \node[red] at (-4.5, 0) {$Ig^0_p$};
    \node[gray] at (-3.35, -1.4) {$\lam=0$};
    \node[gray] at (-1.8, 2) {$\lam=\infty$};
    \node[gray] at (3.3, -1.4) {$\lam=2$};
    \node[gray] at (1.78, -2) {$\lam=1$};

    % --- INTERSECTIONS & LABELS ---

    % Top Oval Intersections (Twisted by pi/10 = 18 degrees)
    \coordinate (T0)   at ({4*cos(18)}, {0.5*sin(18) + 0.8});
    \coordinate (T1)   at ({4*cos(108)}, {0.5*sin(108) + 0.8});
    \coordinate (T2)   at ({4*cos(198)}, {0.5*sin(198) + 0.8});
    \coordinate (Tinf) at ({4*cos(288)}, {0.5*sin(288) + 0.8});

    \filldraw (T0) circle (1.5pt);% node[right, xshift=2pt, yshift=4pt] {$\lambda = 0$};
    \filldraw (T1) circle (1.5pt);% node[above, xshift=-5pt, yshift=5.5pt] {$\lambda = 1$};
    \filldraw (T2) circle (1.5pt);% node[left, yshift=4pt, xshift=-3pt] {$\lambda = 2$};
    \filldraw (Tinf) circle (1.5pt);% node[below right, xshift=4pt, yshift=-2pt] {$\lambda = \infty$};

    % Bottom Oval Intersections (Exactly below top ones)
    \coordinate (B0)   at ({4*cos(18)}, {0.5*sin(18) - 0.8});
    \coordinate (B1)   at ({4*cos(108)}, {0.5*sin(108) - 0.8});
    \coordinate (B2)   at ({4*cos(198)}, {0.5*sin(198) - 0.8});
    \coordinate (Binf) at ({4*cos(288)}, {0.5*sin(288) - 0.8});

    \filldraw (B0) circle (1.5pt);
    \filldraw (B1) circle (1.5pt);
    \filldraw (B2) circle (1.5pt);
    \filldraw (Binf) circle (1.5pt);

    % --- GREEN BULLETS ON Ig_p at j = \infty (Plane z=0) ---
    % Each red oval loops through z=0 twice: at R=5.2 (outer) and R=2.8 (inner)
    \foreach \ang in {18, 108, 198, 288} {
        % Outer intersections
        \filldraw[green!70!black] ({5.2*cos(\ang)}, {0.65*sin(\ang)}) circle (1.5pt);
        % Inner intersections
        \filldraw[green!70!black] ({2.8*cos(\ang)}, {0.35*sin(\ang)}) circle (1.5pt);
    }

\end{tikzpicture}
\label{p-fiber}
\end{equation}

\begin{rem}
(i) As in Example~\ref{compexam}(i), the non-smoothable locus of the fiber $(\ocM_{1,\Gamma})^\Gamma\otimes\FF_p$ is not equidimensional. In fact, this time it is even locally non-equidimensional.

(ii) A set-theoretic fix is simple -- just require that the core $E_0$ is supersingular in the supersingular configurations. However, a scheme-theoretic version requires some care. We will achieve this by taking into account the folding of $\Gamma$ on the semistable core $E_0$, that is, the composition $\Gamma\to E\to E_0$.
\end{rem}

\subsection{Equivariant deformations}
The next stage after classifying the $k$-points of $(\ocM_{1,\Gamma})^\Gamma$ is to study their deformations over $R=k[\veps]/(\veps^2)$. So, choose a $\Gamma$-marked curve $(E,\Gamma_k)$ with $\cha(k)=p$ and let us compute the tangent space to $(\ocM_{1,\Gamma})^\Gamma$ at this curve. By our convention these are tangent spaces over $\ZZ$.

\subsubsection{The basic exact sequence}\label{basicsec}
Let $Z_0$ be the leaf containing $0$ and let $Z_0,Z_1\dots Z_d$ be the chain connecting it to the closest core component $Z_d\subset E_0$. Let $P_1\.P_d$ be the nodes in the chain, that is, $P_j=Z_{j-1}\cap Z_j$, and let $P_{d+1}$ be the node in the core if it is not smooth. Set $\Gamma_{-1}=0$ and let $\Gamma_i$ be the stabilizer of $Z_i$ (in the sense that this is the set of elements taking $Z_i$ to itself). Note that $\Gamma_i$ is also the stabilizer of its node at infinity $P_{i+1}$. With these notation figure (\ref{types1}) is rewritten as follows, where $\gamma\in\Gamma\setminus\Gamma_0$:

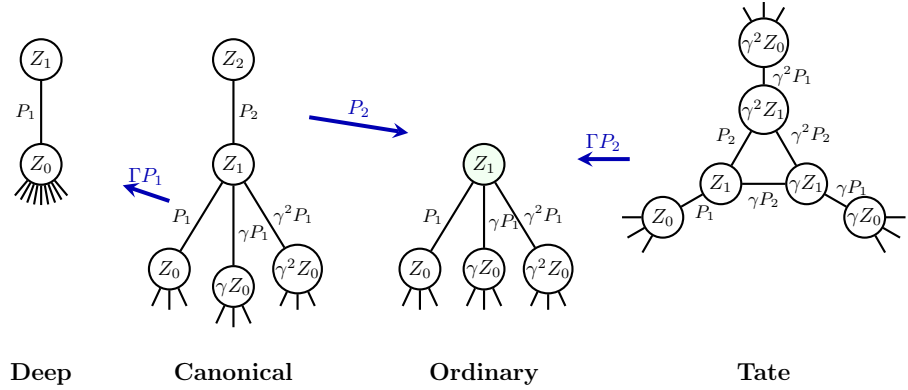
\begin{equation}
\begin{tikzpicture}[scale=0.77, transform shape]

    % Vertical distance and horizontal leaf spread
    \pgfmathsetmacro{\yStep}{1.8}
    \pgfmathsetmacro{\xSpread}{1.1}

    % Explicit X-coordinates for each of the 4 graphs to finely tune spacing
    \pgfmathsetmacro{\xOne}{0}       % Graph 1: Deep
    \pgfmathsetmacro{\xTwo}{3.3}     % Graph 2: Canonical (3.3cm from Deep)
    \pgfmathsetmacro{\xThree}{7.6}   % Graph 3: Ordinary (4.3cm from Canonical)
    \pgfmathsetmacro{\xFour}{12.4}   % Graph 4: Tate (increased by 0.5cm)

    % =========================================================================
    % Type 1: Deep Supersingular (Root E0 -> single Leaf Z, 9 legs)
    % =========================================================================
    \begin{scope}[xshift=\xOne cm]
        \node[node_compressed] (E0_1) at (0, 2*\yStep) {$Z_1$};
        \node[node_compressed] (Z_1) at (0, 1*\yStep) {$Z_0$};

        \draw[edge_thick] (E0_1) -- (Z_1) node[midway, left, edge_label] {$P_1$};

        % 9 legs spread around 270 deg (bottom)
        \foreach \i in {1,...,9}
        {
            \pgfmathsetmacro{\angle}{270 + (\i-5)*10}
            \draw[leg] (Z_1) -- +(\angle:0.7);
        }

        \node[type_label] at (0, -1.8) {Deep};
    \end{scope}

    % =========================================================================
    % Type 2: Canonical Supersingular (Root E0 -> Branch Z -> Leaves Z0...2)
    % =========================================================================
    \begin{scope}[xshift=\xTwo cm]
        \node[node_compressed] (E0_2) at (0, 2*\yStep) {$Z_2$};
        \node[node_compressed] (Z_branch) at (0, 1*\yStep) {$Z_1$};

        \node[node_compressed] (Z0_2) at (-\xSpread, 0) {$Z_0$};
        \node[node_compressed] (Z1_2) at (0, -.3)          {$\gamma Z_0$};
        \node[node_compressed] (Z2_2) at (\xSpread, 0)  {$\gamma^2 Z_0$};

        % Edge l_0
        \draw[edge_thick] (E0_2) -- (Z_branch) node[midway, right, edge_label] {$P_2$};

        % Edges l_1
        \draw[edge_thick] (Z_branch) -- (Z0_2) node[midway, left, edge_label, xshift=-2pt] {$P_1$};
        \draw[edge_thick] (Z_branch) -- (Z1_2) node[pos=0.6, right, edge_label] {$\gamma P_1$};
        \draw[edge_thick] (Z_branch) -- (Z2_2) node[midway, right, edge_label, xshift=2pt] {$\gamma^2 P_1$};

        % 3 legs per leaf (9 total)
        \foreach \i in {1,2,3}
        {
            \pgfmathsetmacro{\angle}{270 + (\i-2)*25}
            \draw[leg] (Z0_2) -- +(\angle:0.7);
            \draw[leg] (Z1_2) -- +(\angle:0.7);
            \draw[leg] (Z2_2) -- +(\angle:0.7);
        }

        \node[type_label] at (0, -1.8) {Canonical};
    \end{scope}

    % =========================================================================
    % Type 3: Ordinary (Root E_0 -> Leaves Z0...2, 3 legs per leaf)
    % =========================================================================
    \begin{scope}[xshift=\xThree cm]
        \node[node_compressed, fill=green!5] (E_3) at (0, 1*\yStep) {$Z_1$};
        \node[node_compressed] (Z0_3) at (-\xSpread, 0) {$Z_0$};
        \node[node_compressed] (Z1_3) at (0, 0)          {$\gamma Z_0$};
        \node[node_compressed] (Z2_3) at (\xSpread, 0)  {$\gamma^2 Z_0$};

        % Edges l_1
        \draw[edge_thick] (E_3) -- (Z0_3) node[midway, left, edge_label, xshift=-2pt] {$P_1$};
        \draw[edge_thick] (E_3) -- (Z1_3) node[pos=0.6, right, edge_label] {$\gamma P_1$};
        \draw[edge_thick] (E_3) -- (Z2_3) node[midway, right, edge_label, xshift=2pt] {$\gamma^2 P_1$};

        % 3 legs per leaf (9 total)
        \foreach \i in {1,2,3}
        {
            \pgfmathsetmacro{\angle}{270 + (\i-2)*25}
            \draw[leg] (Z0_3) -- +(\angle:0.7);
            \draw[leg] (Z1_3) -- +(\angle:0.7);
            \draw[leg] (Z2_3) -- +(\angle:0.7);
        }

        \node[type_label] at (0, -1.8) {Ordinary};
    \end{scope}

    % =========================================================================
    % Type 4: Tate / Bad Reduction (Cycle of length 3)
    % =========================================================================
    \begin{scope}[xshift=\xFour cm]

        % Nested scope to handle the vertical shift cleanly
        \begin{scope}[yshift=1.05*\yStep cm]
            \pgfmathsetmacro{\rCyc}{0.85}        % Radius for the triangle vertices
            \pgfmathsetmacro{\rLf}{2.0}          % Radius for the leaves

            % Triangle Cycle (Vertices: E_0, E_1, E_2)
            \node[node_compressed] (C1) at (90:\rCyc) {$\gamma^2 Z_1$};
            \node[node_compressed] (C2) at (210:\rCyc) {$Z_1$};
            \node[node_compressed] (C3) at (330:\rCyc) {$\gamma Z_1$};

            % Edges of the triangle, labelled 'l'
            \draw[edge_thick] (C1) -- (C2) node[midway, above left, edge_label] {$P_2$};
            \draw[edge_thick] (C2) -- (C3) node[midway, below, edge_label, yshift=-2pt] {$\gamma P_2$};
            \draw[edge_thick] (C3) -- (C1) node[midway, above right, edge_label] {$\gamma^2 P_2$};

            % Leaves (Z_0, Z_1, Z_2)
            \node[node_compressed] (L1) at (90:\rLf) {$\gamma^2 Z_0$};
            \node[node_compressed] (L2) at (210:\rLf) {$Z_0$};
            \node[node_compressed] (L3) at (330:\rLf) {$\gamma Z_0$};

            % Connecting edges
            \draw[edge_thick] (C1) -- (L1) node[midway, right, edge_label, xshift=2pt] {$\gamma^2 P_1$};
            \draw[edge_thick] (C2) -- (L2) node[midway, right, edge_label, xshift=-1pt, yshift=-6pt] {$P_1$};
            \draw[edge_thick] (C3) -- (L3) node[midway, right, edge_label, xshift=-2pt, yshift=6pt] {$\gamma P_1$};

            % Legs (3 per leaf, projecting outwards)
            \foreach \i in {1,2,3} {
                \pgfmathsetmacro{\angle}{90 + (\i-2)*30}
                \draw[leg] (L1) -- +(\angle:0.7);
            }
            \foreach \i in {1,2,3} {
                \pgfmathsetmacro{\angle}{210 + (\i-2)*30}
                \draw[leg] (L2) -- +(\angle:0.7);
            }
            \foreach \i in {1,2,3} {
                \pgfmathsetmacro{\angle}{330 + (\i-2)*30}
                \draw[leg] (L3) -- +(\angle:0.7);
            }
        \end{scope}

        % Label correctly aligned with the other three graphs
        \node[type_label] at (0, -1.8) {Tate};
    \end{scope}

    % =========================================================================
    % CONTRACTION ARROWS
    % =========================================================================
    % Note: The stealth- style draws the arrow tip at the FIRST coordinate.

    % Arrow from Type 2 to Type 1 (Contracting bottom edges l_1)
    % Points leftwards from Canonical to Deep
    \draw[contraction_arrow]
        (\xOne + 1.4, 0.8*\yStep) -- (\xTwo - 1.1, 0.65*\yStep)
        node[midway, above, contraction_text] {$\Gamma P_1$};

    % Arrow from Type 2 to Type 3 (Contracting upper edge l_0)
    % Points rightwards from Canonical to Ordinary
    \draw[contraction_arrow]
        (\xThree - 1.3, 1.3*\yStep) -- (\xTwo + 1.3, 1.45*\yStep)
        node[midway, above, contraction_text] {$P_2$};

    % Arrow from Type 3 to Type 4 (Contracting cycle edges l)
    % Points rightwards from Ordinary to Tate
    \draw[contraction_arrow]
        (\xThree + 1.6, 1.05*\yStep) -- (\xFour - 2.3, 1.05*\yStep)
        node[midway, above, contraction_text] {$\Gamma P_2$};
\end{tikzpicture}
\label{types3}
\end{equation}

Since $\Gamma$ acts transitively on the layers of $E$, each orbit of the $\Gamma$-action on the sets of irreducible components and nodes contains a unique representative of the form $Z_i$ and $P_j$, respectively. In addition we note that $\Gamma_i$ acts on rational components through $\GG_a$ and hence acts trivially on the tangent space $T_{P_i}(Z_{i-1})$. Also it acts trivially on $Z_i$ and hence acts trivially on $T_{P_i}$. Thus the general theory recalled in Appendix, see \S\ref{summary} and (\ref{summaryeq}), yields the exact sequence, where $D_i$ is the set of marked points and nodes in $Z_i$, and $\nu\in\{0,1\}$ with $\nu=1$ in Tate's case.
\begin{equation}\label{basicseq}
0\to\oplus_{i=0}^d T_{(Z_i,D_i)}(\ocM_{g_i,n_i})^{\Gamma_i/\Gamma_{i-1}}\to T_{(E,\Gamma)}(\ocM_{1,\Gamma})^\Gamma\to\oplus_{j=1}^{d+\nu} T_{P_j}\stackrel{\partial}{\to}\dots
\end{equation}

Note that $g_i=0$ for $i<d$ and $g_d\le 1$ with equality if and only if the core is an elliptic curve. Also $\Gamma_i/\Gamma_{i-1}$ fixes the node at infinity and acts transitively on the set of finite nodes (or marked points if $i=0$). In particular, $n_i=|D_i|=|\Gamma_i/\Gamma_{i-1}|+1$.

\subsubsection{Smoothing the nodes}
A priori the connecting homomorphism $\partial$ provides an obstacle to an equivariant smoothing of the node over $\FF_p$. It will follow from
computations of deformation spaces (see Lemma~\ref{mainlem2} below) that at least in the smoothable case $\partial=0$ and hence different types of nodes can be smoothed independently. 

\begin{exam}\label{partialexam}
At this stage, for the sake of illustration we can easily show that $\partial$ has a non-trivial kernel whenever $(E,\Gamma)$ is a smoothable $\Gamma$-marked $k$-curve and the case is ramified: $p=\cha(k)$ is odd or the configuration is supersingular. Indeed, by Lemma~\ref{ramlem} any smoothing $(\calE,\Gamma_R)$ of $(E,\Gamma)$ over a mixed characteristic DVR $R$ is ramified over $p$, and hence induces a smoothing deformation over $\FF_p[\veps]/(\veps^2)$. The class of this deformation maps to a non-zero linear combination of types of nodes which lies in the kernel of $\partial$. In particular, if such a configuration contains only a single type $P$ of nodes (the ordinary case when $p$ is odd or the deep supersingular case), then $\partial(T_P)=0$.
\end{exam}

\subsubsection{Contribution of connected components}
Now, let us compute dimensions of the terms connected components contribute to the sequence (\ref{basicseq}).

\begin{lem}\label{ratcomp}
Let $k$ be an algebraically closed field of characteristic $p$, and let $\veps_p\in\{0,1\}$ be the parity of $p$.

(i) Let $\Gamma_0=\ZZ/p\ZZ$ act on $Z=\PP^1_k$ through $\FF_p$ by translations and let $\oGamma_0=\FF_p\cup\{\infty\}$. Then $\dim(H^1(Z,T_Z(-\oGamma_0))^{\Gamma_0})=\veps_p$.

(ii) Let $\Gamma=(\ZZ/p\ZZ)^2$ act on $Z=\PP^1_k$ through $\FF_p\oplus\lam\FF_p$ with $\lam\in k\setminus\FF_p$ by translations and let $\oGamma=\FF_p\oplus\lam\FF_p\cup\{\infty\}$. Then $\dim(H^1(Z,T_Z(-\oGamma))^\Gamma)=1+\veps_p$.

(iii) Let $\Gamma_0$ act on an ordinary elliptic $k$-curve $E$ through $\Gamma_0=E[p](k)$. Then $\dim(H^1(E,T_E(-\Gamma_0))^{\Gamma_0})=1$.
\end{lem}
\begin{proof}
(i) The spaces $V=H^0(Z,\Omega^{\otimes 2}_Z(\oGamma_0))$ and $W=H^1(Z,T_Z(-\oGamma_0))$ are dual. The elements of $V$ are of the form $h(t)dt^{\otimes 2}$, where $h(t)\in k[t,(t^p-t)^{-1}]$ has at most simple poles at $\FF_p$ and zero of order at least 3 at infinity, that is, $h=f/(t^p-t)$ with $\deg(f)\le p-3$. If $h$ is $\Gamma_0$-invariant, $f$ is also $\Gamma_0$-invariant, and hence it is constant. Therefore $\dim(V^{\Gamma_0})=\veps_p$ with the invariant generator $\frac{dt^{\otimes 2}}{t^p-t}$ when $p>2$. Using that $k[\Gamma_0]=k[\veps]/(\veps^p)$, where $\veps=1-\gamma$ and $\gamma\in\Gamma_0$ is a generator, and $V^{\Gamma_0}=\{v\in V|\ \veps v=0\}$, we obtain that $V$ is a cyclic $k[\Gamma_0]$-module (of dimension $p-2$). Hence its dual $W$ is also cyclic and we obtain that $\dim(W^{\Gamma_0})=\veps_p$.

(ii) As above, the spaces $V=H^0(Z,\Omega^{\otimes 2}_Z(\oGamma))$ and $W=H^1(Z,T_Z(-\oGamma))$ are dual, and this time $\dim(V)=p^2-2$ is positive also for $p=2$. Thus $V^\Gamma\neq 0$ and we claim that $\dim(V^\Gamma)=1$. Indeed, any element $f\in V$ has a simple pole at an element of $\FF_p\oplus\lam\FF_p$, hence by the transitivity of the action any element of $V^\Gamma$ has simple poles at all $p^2$ points of $\FF_p\oplus\lam\FF_p$. But if $V^\Gamma$ contains linearly independent $v,w$ there would exist a non-zero $v+aw$ which has no pole at 0.

The contradiction shows that $\dim(V^\Gamma)=1$, but this time $\Gamma$ is not cyclic, so $V_\Gamma$ (and its dual $W^\Gamma$) may have another dimension. Fortunately, one can still avoid explicit computations as follows. The group ring $R=k[\Gamma]=k[\veps,\nu]/(\veps^p,\nu^p)$ is lci, and $k=V^\Gamma=\{v\in V|\ m_Rv=0\}$ is the socle of $V$. It follows that the injective hull of $V$ is the injective hull of $k=R/m_R$, which is $R$ itself. Thus $V\into R$ as modules, and using that $\dim(R)-\dim(V)=2$ we obtain that 1 and $m=m_R$ are not contained in $V$, so $V$ is spanned by $m^2$ and an element $v\in m$. If $p=2$, then $m^2$ is generated by a single element and lies in $mv$, hence $V$ is generated by $v$. If $p>2$, then $m^2/m^3$ is three-dimensional and $mv/m^2v$ is its two-dimensional subspace, and hence $V$ is generated by $v$ and an element of $m^2$. All in all, $V_\Gamma=V/mV$ is of dimension $1+\veps_p$, and hence the same is true for its dual $W^\Gamma$.

(iii) The isogeny $E\to E'=E/\Gamma_0$ is \'etale of degree $p$ and $O'$ is the image of $\Gamma_0$. The pullback induces an isomorphism $H^1(E',T_{E'}(-O'))=H^1(E,T_{E}(-\Gamma_0))^{\Gamma_0}$, and the left hand side space is dual to $H^0(E',\cO_{E'}(O'))=k$.
\end{proof}

\subsubsection{Dimension count and interpretations}\label{dimcount}
%Now we have enough info on contribution of all terms and can make some conclusions. First we can deduce that the first order deformations are always unobstructed.

%\begin{cor}\label{partialcor}
%If $(E,\Gamma)$ is a smoothable $\Gamma$-marked $k$-curve, then $\partial(T_P)=0$ for any node $P\in E$.
%\end{cor}
%\begin{proof}
%We should consider the only case not covered by Lemma~\ref{partiallem}, when $p=2$ and the configuration is ordinary or Tate. By sequence (\ref{basicseq}) and Lemma~\ref{ratcomp} the tangent space over $\FF_2$ is at most two-dimensional because $\veps_p=0$ and the only contributions are from the core component and the node, and both are at most one-dimensional. Therefore the absolute cotangent space to $(\ocM_{1,\Gamma})^\Gamma$ at $x=(E,\Gamma)$ is at most three-dimensional. On the other hand, $(E,\Gamma)$ can be lifted in characteristic zero both to a point in $X(2)$ and in the non-semistable Igusa or Tate component on diagram (\ref{types2}). Therefore, it lies in the intersection of (at least) two two-dimensional components, and the cotangent space is at least three-dimensional. Thus, the absolute tangent space is three-dimensional and it contains a non-trivial contribution of the node. In particular, $\partial(T_P)=0$.
%\end{proof}

We already have enough information to show that $(\ocM_{1,\Gamma})^\Gamma$ is quite singular along the $\FF_p$-fiber.

\begin{rem}\label{dimcountrem}
Let $p$ be odd and let $(E,\Gamma)$ be a $\Gamma$-marked $k$-curve with a single type of nodes.

(i) Assume the configuration is ordinary. By Example~\ref{partialexam} and Lemma~\ref{ratcomp} the deformation space to $(\ocM_{1,\Gamma})^\Gamma$  at $(E,\Gamma)$ is three-dimensional because the two types of components and the node contribute one dimension. This fits our knowledge that $(\ocM_{1,\Gamma})^\Gamma$ is not normal along the ordinary locus because it lies in the intersection of the closure of the generically smooth and Igusa components from figure (\ref{types1}).

(ii) In the same manner one checks that the space of deformations in the deep supersingular case is at least of dimension $1+2\veps_p=3$ combined from the contributions of the two components. This component of $(\ocM_{1,\Gamma})^\Gamma$ is generically non-liftable to the mixed characteristic, hence it is generically non-reduced.
\end{rem}

\subsubsection{Conclusions}
To cut out $\cX(p)$ from $(\ocM_{1,\Gamma})^\Gamma$ one has to solve two issues: (1) remove the supersingular configurations with an ordinary core, (2) eliminate the non-smoothable components and non-reduced structure due to existence of parasitic equivariant deformations of rational components. The first task will be easily solved by use of the retraction onto the core (or $\phi_p$). The second task is more delicate and will be solved by imposing a compatibility condition with certain vector fields of a global nature. This seems to be unavoidable because the non-smoothable deformations which are ``parasitic'' for our goals do have a natural local interpretation as we will see in the next subsection.

\subsection{Fixed loci}\label{fixedloci}
We can now describe the structure of a few basic $\Gamma$-fixed and $\Gamma_0$-fixed loci, where we fix the notation $\Gamma_0=\ZZ/p\ZZ$ throughout this section. The examples in genus 0 are quite instructive, illustrate some phenomena and can be viewed as simple basic blocks from which $(\ocM_{1,\Gamma})^\Gamma$ is built on the level of deformations.

Our arguments are based on combinatorial descriptions of the sets of geometric points and computation of the tangent spaces via Lemma~\ref{ratcomp}, and we will not try to figure out the precise non-reduced structure beyond what is easily obtained without computations.

\subsubsection{$\oGamma_0$-marked rational curves}
Set $\oGamma_0=\Gamma_0\cup\{\infty\}$ and let $\Gamma_0$ act by translations (so, $\infty$ is fixed).

\begin{lem}\label{ratGamma0}
The fixed locus scheme $S=(\ocM_{0,\oGamma_0})^{\Gamma_0}$ is naturally isomorphic to $\Spec(\ZZ[\xi_p])$.% The universal curve is $\PP^1$ with sections $\frac{1-t}{t(1-\xi_p)}$ for $t\in\{0\}\cup\mu_p$.
\end{lem}
\begin{proof}
First, let us show that the geometric points of $S$ are as claimed. Let $(C,\oGamma_0)$ be a $k$-point. Since $C$ is of genus zero and $\Gamma_0$ is a simple group, it is easy to see that any $\Gamma_0$-stable configuration is irreducible. Thus $C=\PP^1_k$ and we can assume that $\oGamma_0\into\PP^1_k$ sends infinity to infinity, and hence $\Gamma_0$ acts through $\GG_a(k)\rtimes\GG_m(k)$. The action is through $\GG_a(k)$ when $\cha(k)=p$ and through $\GG_m(k)$ otherwise. In the first case, we can choose the coordinate so that $0,1\in\Gamma_0$ are mapped to $0,1\in \PP^1_k$, and hence $S(k)$ is a singleton parameterizing the configuration $\oGamma_0=\FF_p\cup\{\infty\}\into\PP^1_k$. In the second case, $\Gamma_0$ acts on the tangent space at infinity via a character $\Gamma\toisom\mu_p\subset k^\times$ and this character is an invariant of the conjugacy class. In fact, we have chosen a coordinate so that the second fixed point of the action is at the origin, and $0\in\Gamma_0$ is mapped to $1\in\PP^1_k$, while the image of the generator $1\in\Gamma_0$ in $\mu_p\subset\PP^1_k$ determines the configuration. Alternatively, one can spell this as follows: $S(k)$ consists of $p-1$ points which parameterize $\PP^1_k$ with marked points $\oGamma_0\toisom\mu_p\cup\{\infty\}$ and the action is through a given character $\Gamma_0\toisom\mu_p$.

Since $S=\Spec(R)$ is finite over $\ZZ$, this already implies that $\Spec(\ZZ[\xi_p])$ is the normalized reduction of $S$. It remains to note that $S$ is regular because $S\otimes\ZZ[1/p]$ is \'etale over $\Spec(\ZZ[1/p])$, and it follows from Lemma~\ref{ratcomp}(i) that the cotangent space $m_x/m_x^2$ at the $\FF_p$-point $x\in S$ is one-dimensional. Indeed, for $p=2$ the tangent space to $S\otimes\FF_2$ is trivial, and hence $m_x=(2)$, and if $p>2$, then the tangent space to $S\otimes\FF_p$ is one-dimensional, and hence $m_R=(p,\veps)$, where $\veps$ is nilpotent in $R/pR$. So, for $p>2$ we also have that $m_R=(\veps)$ is principal. In fact, $m_R$ is always generated by $\veps=1-\xi_p$, regardless to the parity of $p$.
\end{proof}

In particular, we see that $(\ocM_{0,\oGamma_0})^{\Gamma_0}\otimes\FF_p=\Spec(\FF_p[\veps]/(\veps^{p-1}))$ is non-reduced for $p>2$, as we already observed in \S\ref{dimcount}. However, the deformations which are non-smoothable over $\FF_p$ can be lifted to a regular (but non-smooth) $\ZZ$-scheme, and the multiplicative deformations of the action, which looked pathologically in characteristic $p$, naturally assemble into the deformation to zero characteristic.

\subsubsection{$\oGamma$-marked rational curves}
Next, let us describe the fixed moduli locus of the whole $\Gamma$ acting on rational curves through the set $\oGamma=\Gamma\cup\{\infty\}$ of marked points.

\begin{lem}\label{ratGamma}
The fixed locus scheme $S=(\ocM_{0,\oGamma})^{\Gamma}$ is purely one-dimensional and consists of $p+1$ clusters $S_\lam\toisom\Spec(\ZZ[\xi_p]\otimes\ZZ[\xi_p])$, where $\lam\in\PP^1(\FF_p)$, of $p-1$ horizontal components isomorphic to $\Spec(\ZZ[\xi_p])$, and a vertical component $V$ with reduction $\PP^1_{\lam,\FF_p}$ meeting the clusters at the $p+1$ points of $\PP^1(\FF_p)$. If $p>2$, then $V$ is generically non-reduced and the nil-radical on $\PP^1_\lam\setminus\PP^1(\FF_p)$ is invertible, and if $p=2$, then $V$ is smooth at any point $\lam\notin\PP^1(\FF_p)$.
\end{lem}
\begin{proof}
If $\cha(k)\neq p$, then there is no faithful action of $\Gamma$ on $\PP^1_k$, and any stable $\Gamma$-equivariant curve $(C,\oGamma_0)$ of genus zero consists of a core $C_0=\PP^1$ with $p$ leaves $Z_1\.Z_p$ and $p$ marked points on each leaf (an ordinary configuration but with core of genus zero). The stabilizer $\Gamma_0$ of leaves is a cyclic subgroup of $\Gamma$ of order $p$, hence there exists $p+1$ choices of (clusters of) components parameterized by $\lam\in\PP^1(\FF_p)$. Fix the choice of $\Gamma_0$ and let us describe the cluster $S_\lam$. Precisely as in the proof of Lemma~\ref{ratGamma0}, the action is determined by the characters $\Gamma/\Gamma_0\toisom\mu_P$ corresponding to the action at $\infty\in C_0$ and the character $\Gamma_0\toisom\mu_p$ corresponding to the action at the nodes of $Z_i$ (these characters coincide because of the $\Gamma$-action). This proves that $S\otimes\ZZ[1/p]=\coprod_\lam S'_\lam$ where each cluster $S'_\lam$ is of the form $\Spec(\ZZ[1/p,\xi_p]\otimes\ZZ[1/p,\xi_p])$ and hence is a disjoint union of $p-1$ copies of $\Spec(\ZZ[1/p,\xi_p])$.

If $\cha(k)=p$, then there are two types of configurations: the ordinary one with one core and $p$ leaves and the non-liftable deep supersingular one with $C=\PP^1_k$ acted on by $\Gamma=\FF_p\oplus\lam\FF_p$. Each ordinary cluster over $\FF_p$ consists of a single point $P_\lam$ because one can choose coordinates on the core and the leaves so that $\Gamma/\Gamma_0$ and $\Gamma_0$ are identified with $\FF_p$. The supersingular cluster is parameterized by $\lam\in\PP^1(k)\setminus\PP^1(\FF_p)$. Thus, on the set-theoretic level $S$ consists of horizontal clusters $S_\lam$, a vertical component $V$ with reduction $\PP^1_{\lam,\FF_p}$ and the intersection of $S_\lam$ with $V$ occurs at the point $P_\lam$.

The claim about the nilradical for $p>2$ and smoothness for $p=2$ follows from the fact that the relevant tangent spaces are of dimension $1+\veps_p$ by Lemma~\ref{ratcomp}(ii), and the configurations with $\lam\notin\PP^1(\FF_p)$ are non-liftable to mixed characteristic.
\end{proof}

\subsubsection{$\Gamma_0$-marked elliptic curves}
Next, we study $\Gamma_0$-marked elliptic curves. From the classical theory we know that over $\ZZ[1/p]$ there exists at least the component  $X_1(p)$ which parameterizes generalized elliptic curves with a fixed $p$-torsion point. We should describe how it is extended to $\FF_p$ (this time we will just recover the Katz-Mazur compactification) and what other components are. We assume that $p>3$ to avoid dealing with Picard invariance conditions.

\begin{lem}\label{ellipticgamma_0}
Assume that $p>3$. Then the fixed locus $S=(\ocM_{1,\Gamma_0})^{\Gamma_0}$ is purely two-dimensional and consists of two irreducible components:
\begin{itemize}
\item[(1)] The Katz-Mazur model $S_1=\tilcX_1(p)$, whose $\FF_p$-fiber consists of the Igusa curve $\Ig_p$ parameterizing \'etale level structures, and a non-reduced canonical component $\PP^1_{j,\FF_p}$ parameterizing Drinfeld level structures of multiplicity $p$.
\item[(2)] The projective line $S_2=\PP^1_{j,\ZZ[\xi_p]}$ parameterising configurations with an irreducible core and one leaf containing all marked points.
\end{itemize}
The components $S_1$ and $S_2$ intersect along $\PP^1_{j,\FF_p}$.
\end{lem}
\begin{proof}
We start, again, with the geometric points. The $\FF_p$-fiber has two irreducible components -- the Igusa curve $\Ig_p$ parameterizing an ordinary generalized elliptic curve $E$ provided with a level structure $\Gamma_0=E[p](k)$, and the component $\PP^1_{j,\FF_p}$ parameterizing a configuration of a generalized elliptic curve $E_0$ serving as a core and a rational leaf with $\FF_p$ marked points. This components intersect at the supersingular points: the core is supersingular and the marked points sit on a leaf. Thus, on the level of reductions one obtains the $\FF_p$-fiber of the Katz-Mazur model $\tilcX_1(p)$, and hence the whole locus is smoothable.

Now let us consider the situation over $\ZZ[1/p]$. There exist two connected components -- the usual $X_1(p)$ parameterizing $\Gamma_0$-marked generalized elliptic curves, and a component $S_2$ parameterizing an irreducible core $E_j$ with a leaf $Z$ containing marked points. As in the previous lemmas, the marked points can be arranged in a $\mu_p$ configuration, and the character of the action at the node of $Z$ is an invariant of the configuration. Therefore $S_2\otimes\ZZ[1/p]=\PP^1_{j,\ZZ[1/p,\xi_p]}$, and it remains to recall that $X_1(p)$ specializes to the whole $\FF_p$-fiber, while $S_2$ specializes to $\PP^1_{j,\FF_p}$ only.
\end{proof}

This time the moduli space is non-normal, and the ``parasitic'' deformations in the $\PP^1_{j,\FF_p}$ locus extend to the mixed characteristic so that the curve does not smoothen and the $\FF_p$-action on the leaf deforms to a $\mu_p$-action. The Igusa curve is reduced in the fiber by the results of Katz-Mazur, and this matches the computation of tangent spaces in Lemma~\ref{ratcomp}(iii).

\subsubsection{$\Gamma$-marked elliptic curves}
The case of our interest $(\ocM_{1,\Gamma})^\Gamma$ combines all types of the above behaviours -- different components intersect along the $\FF_p$-fiber, and there also is a non-reduced non-smoothable vertical component. We refer to figure (\ref{types2}) for types of non-smoothable configurations.

\begin{lem}\label{ellipticgamma}
Assume that $p>3$. Then the fixed locus $S=(\ocM_{1,\Gamma})^{\Gamma}$ is purely two-dimensional and consists of four clusters of components of different types:
\begin{itemize}
\item[(i)] The smoothable horizontal component $S_1=\cX(p)$.
\item[(ii)] The moduli space of Igusa configurations $S_2=\coprod_{\lam\in\PP^1(\FF_p)} S_{2,\lam}$ consisting of $p+1$ horizontal components $S_{2,\lam}\toisom\tilcX_1(p)\otimes\ZZ[\xi_p]$.
\item[(iii)] The moduli space of canonical configurations $S_3=\coprod_{\lam\in\PP^1(\FF_p)}S_{3,\lam}$ consisting of $p+1$ clusters $S_{3,\lam}\toisom\PP^1_{j}\otimes\ZZ[\xi_p]\otimes\ZZ[\xi_p]$ of $p-1$ horizontal components $\PP^1_{j,\ZZ[\xi_p]}$.
\item[(iv)] A vertical component $S_4$ with reduction $\PP^1_{j,\FF_p}\times\PP^1_{\lam,\FF_p}$. 
\end{itemize}
The components only intersect in the $\FF_p$-fiber, which is the union of $S_4$ and $p+1$ Igusa curves. The reduction of $S_1\otimes\FF_p$ is the union of $p+1$ Igusa curves and the supersingular locus $S_p(j)=0$, the reduction of $S_2\otimes\FF_p$ is the union of $p+1$ Igusa curves and $p+1$ curves $\PP^1_{j,\FF_p}$, and the reduction of $S_3\otimes\FF_p$ is the union of $p+1$ curves $\PP^1_{j,\FF_p}$.
\end{lem}
\begin{proof}
We know from Lemma~\ref{Gammasmoothlem} that there is one vertical (i.e. non-smoothable) irreducible component $S_4$ with reduction $\PP^1_{j,\FF_p}\times\PP^1_{\lam,\FF_p}$, and at least when $p>2$ it is generically non-reduced by Remark~\ref{dimcountrem}(ii). There is also one smoothable component $S_1=\cX(p)$. Other components are horizontal, hence correspond to non-smoothable classification on figure (\ref{types2}): canonical and Igusa configurations. These two cases are classified by the same method as was used in the proof of the previous lemmas, so we will not repeat the details.

The intersection $S_1\cap S_4$ is described by Lemma~\ref{Gammasmoothlem}, and the degenerations of $S_2$ and $S_3$ modulo $p$ are easier to study and we skip the details.
\end{proof}

In particular, the components $S_1$ and $S_4$ intersect only in codimension 2.

\section{A functorial description of $\cX(p)$}\label{4sec}
Our next goal is to introduce the additional conditions which cut out $\cX(p)$ from $(\ocM_{1,\Gamma})^\Gamma$.

\subsection{Contraction onto the core}\label{phisec}

\subsubsection{Drinfeld level structures}
Recall that Drinfeld level-$N$ structure on an elliptic $S$-curve $E$ is a homomorphism $\ell\:\Gamma_S=(\ZZ/N\ZZ)^2_S\to E[N]\into E$ which generates the image as a Cartier divisor, that is, the ideal defining $E[N]$ is the product of the ideals defining $\ell(\gamma)$ for $\gamma\in\Gamma$. The main example to have in mind is of an elliptic curve $\cE$ over a mixed characteristic DVR $(R,m)$. The subscheme $\cE[p]$ is the vanishing locus of the multiplication-by-$p$ polynomial $\phi_p$. The Drinfeld level structure is the decomposition $\phi_p(t)=\prod_{i=1}^{p^2}(t-\gamma_i)$ over the elements $\gamma_i\in\cE[p](R)$. Over $K$ all roots are distinct and it is equivalent to the usual level structure. On the other hand,  $\cE\otimes R/m^a$ has no usual level structure, while its Drinfeld level structure is obtained by reducing the same decomposition modulo $m^a$. In particular, it book keeps the information whether some roots collide modulo $m^a$ or not.

\subsubsection{Compatibility of level structures}
Let $(E,\Gamma_S)$ be a $\Gamma$-marked genus-1 $S$-curve. The contraction onto the core $E\to E_0$ is a $\Gamma$-equivariant morphism which induces a morphism $\Gamma_S\to E_0$. The condition that $\Gamma_S\to E_0$ is a Drinfeld level structure on $E_0$ will be denoted ($\phi_p$). It defines a closed subscheme of $(\ocM_{1,\Gamma})^\Gamma$ that will be denoted $(\ocM_{1,\Gamma})^{\Gamma,\phi_p}$.

\begin{rem}
The notation will be essential in the sequel paper when levels $p^n$ are studied. For level $p$ it just indicates compatibility with the multiplication-by-$p$ morphism $\phi_p\:E_0\to E_0$. Namely, $\Gamma_S\to E_0$ is the pullback of $O_S\into E_0$ under $\phi_p$ in the sense of Cartier divisors (rather than morphisms of schemes).
\end{rem}

Imposing the Drinfeld level condition on $(\ocM_{1,\Gamma})^\Gamma$ one gets rid of non-smoothable components and cuts out the maximal locus in $(\ocM_{1,\Gamma})^{\Gamma}$ which admits a natural morphism to $\tilcX(p)$.

\begin{lem}\label{philem}
(i) Contraction onto the core induces a morphism $(\ocM_{1,\Gamma})^{\Gamma,\phi_p}\to\tilcX(p)$ onto the Katz-Mazur model.

(ii) The closed immersion $\cX(p)\into(\ocM_{1,\Gamma})^{\Gamma,\phi_p}$ is bijective.
\end{lem}
\begin{proof}
The first claim is obvious. To prove (ii) we should show that any non-smoothable $\Gamma$-marked $k$-curve $(E,\Gamma_k)$ does not satisfy the Drinfeld level structure condition. If $\cha(k)=p$, then the curve has an ordinary core and a supersingular configuration, see Lemma~\ref{Gammasmoothlem}. Such a configuration induces the map $\Gamma_k\to E_0$ of multiplicity $p^2$, which is not a Drinfeld level structure when $E_0$ is not supersingular. In the same way, if $\cha(k)\neq p$, then Drinfeld level structures are \'etale, but any non-smoothable configuration from diagram (\ref{types2}) induces a map $\Gamma_k\to E_0$ at least of multiplicity $p$.
\end{proof}

\begin{rem}\label{phirem}
Although, ($\phi_p$) determines $\cX(p)$ set-theoretically, this is not so on the level of schemes. For example, if $p>2$ and $(E,\Gamma)$ is ordinary over $k$ of characteristic $p$, then the space of deformations of $\Gamma$-marked curves was described in Remark~\ref{dimcountrem}(i), and there exists a non-smoothable deformation $(\cE,\Gamma_R)$ over $R=k[\veps]/(\veps^2)$ in the direction of the Igusa component over $\ZZ$ which does not smoothen the node. This deformation satisfies $(\phi_p)$ because the maps $\gamma\to\cE(R)\to\cE_0(R)$ depend only on the leaf of $\cE$ to which $\gamma$ is mapped, so $\Gamma\to\cE_0(R)$ has $p$ components of multiplicity $p$ and hence it coincides with the Drinfeld level structure of the ordinary elliptic $R$-curve $\cE_0$.

This can be interpreted as follows: $(\phi_p)$ does eliminate the Igusa component over $\ZZ[1/p]$, but it preserves a ramified infinitesimal part which lives over $\FF_p$.
\end{rem}

\subsection{Scaled vector fields and affine actions}
To eliminate the remaining non-reduced structure on $(\ocM_{1,\Gamma})^{\Gamma,\phi_p}$ we will need to impose one more restriction on the action that will be called affinity.

\subsubsection{Scaled reduction}\label{scaledsec}
Let $S$ be an integral scheme, $E$ a prestable genus-1 $S$-curve with a smooth generic fiber and $Z$ an irreducible component of a fiber $E_s$. Consider a non-zero invariant vector field $\partial_\eta\in\Gamma(T_{E_\eta})$ on the generic fiber. Multiplying $\partial_\eta$ by a non-zero element of $\cO_s$ we can assume that it extends to a vector field $\partial\in T_{E/S,\veps}$, where $\veps$ is the generic point of $Z$. Since the sheaf $T_{E/S}$ is invertible at $\veps$, we have that $\partial=a\partial_Z$, where $a\in\cO_s$ and $\partial_Z\in T_{E/S,\veps}$ is a vector field which is defined in a neighborhood of $\veps$ and does not vanish there. Both $a$ and $\partial_Z$ are unique up to a unit $u\in\cO^\times_s$, and one can view $(a)$ as the vanishing divisor of $\partial$ at $\veps$. We call $\partial_Z$ a {\em scaled reduction} of $\partial$ along $Z$.

\subsubsection{Independence of choices}
We will now use smoothings to extend the above construction to any base $S$.

\begin{lem}\label{scaledder}
Let $E$ be a prestable genus-1 curve over a scheme $S$ and let $Z$ be an irreducible component of a fiber $E_s$ with the generic point $\veps$. Then up to a unit $u\in\cO^\times_S$ there exists a unique way to choose a vector field $\partial_Z\in T_{E/S,\veps}$ so that the choice is compatible with \'etale covers and scaled reductions of generically regular vector fields on smoothings of $E/S$.
\end{lem}
\begin{proof}
By \'etale descent we can work \'etale locally on $S$. In particular, we can assume that a smoothing $S\into T$, $\cE/T$ exists and the scaled reduction provides a candidate for $\partial_Z$. The only thing we need to check is that this construction is independent of choices. Since up to the further \'etale base change the set of smoothings is filtered by Lemma~\ref{smoothabilitylem}, we can assume that $T\into T'$, $\cE\into\cE'$ is a larger smoothing and our task is to prove that the corresponding scaled reductions $\partial_Z$ and $\partial'_Z$ coincide up to a unit.

Let $\eta\in T$ and $\eta'\in T'$ be the generic points. Take any non-vanishing derivation $\partial_{\eta'}$ on the generic fiber of $\cE'$ which extends to $T_{\cE'/T',\veps}$ and factor it as $\partial_{\eta'}=a\partial'$ with $\partial'$ a scaled reduction at $\veps$ induced from $\partial'$. Since $\partial'$ does not vanish at $\veps$ it also does not vanish at the generic point of $\cE_\eta$, which is its generization. Therefore, $\partial'$ restricts to a regular vector field on $\cE_\eta=\cE'_\eta$ and hence it is also a scaled reduction obtained from the smoothing $\cE$.
\end{proof}

\begin{rem}
An important particular case is obtained when $S$ is a {\em fat point}, that is, a scheme whose underlying topological space is a point. In this case, one can view $\partial_Z$ as a meromorphic vector field on $Z$. In fact, $\partial_Z$ is even a regular derivation, as follows from the lemma below, but it is not a logarithmic derivation and does not vanish at finite nodes of $Z$.
\end{rem}

\subsubsection{Scaled vector fields}
Any derivation $\partial_Z$ as in Lemma~\ref{scaledder} will be called a {\em scaled vector field}. Naturally, there is a more intrinsic definition which does not use smoothings, but it will be studied elsewhere. %We postpone it to \S\ref{groupsec}. 
In the sequel we will only need the following simple concrete computation.

\begin{lem}\label{scaleder2}
Let $E$ be a prestable genus-1 curve over a fat point $S=\Spec(R)$, and assume that $Z=\PP^1_S$ is a geometrically irreducible rational component with a coordinate $t$ such that $t=\infty$ defines the node $P$ which separates $Z$ from the core. Also, let $\pi\in R$ denote the total modulus between $Z$ and the core (i.e the product of the moduli of all nodes on the chain between $Z$ and the core). If $\pi=0$, then $\partial_t$ is a scaling derivation on $Z$.
\end{lem}
\begin{proof}
The claim can be checked \'etale locally, so we can fix a smoothing $S\into T=\Spec(A)$, $\cE/T$, where $E=\cE\times_TS$ and $A$ is a local domain. The stabilization of $\cE$ onto $Z$ (see \S\ref{stabcompsec}) keeps the generic fiber $\cE_\eta$ unchanged, and contracts the closed fiber onto the curve consisting of two components: the crown $Z$ and the core component $E_0$ closest to $Z$. The retraction is an isomorphism over an open subscheme containing the generic fiber $\cE_\eta$ and the generic point of $Z$, hence it does not modify scaled vector fields on $Z$. In addition, the retraction does not modify the total modulus $\pi$ by Lemma~\ref{stablem}(iv), so we can replace $\cE$ with the stabilization and assume from now on that $E$ is obtained from an irreducible generalized elliptic curve $(E_0,O_S)$ and a leaf $Z$ by identifying $O_S$ with the infinity of $Z$.

The simplest way to proceed is to re-choose the smoothing. First, lift $E_0$ to a generically smooth generalized elliptic curve $(\cE_0,O_T)$ and choose a non-zero $\pi\in A$ whose image vanishes in $R$. Then consider the blowup $\cE\to\cE_0$ along the ideal generated by $\pi$ and the ideal $\cI$ defining $O_T$ in $\cE_0$. This does not modify the generic fiber and inserts a new component $Z$ over the closed point $s$. More concretely, if $x\in\cO_{\cE_0,O_s}$ is a generator of $\cI$, then the blow up is along $(x,\pi)$ and the $\pi$-chart contains a new leaf $Z$ with the coordinate $t=x/\pi$. In particular, the modulus of the created node is $\pi$, and hence $\cE$ is a smoothing of $E$.

Finally, choose an invariant derivation $\partial$ on $\cE_0$. After normalization we can assume that locally at $O_s$ it is of the form $\partial_x+ux\partial_x$ for some $u\in\cO_{\cE_0,O_s}$. Therefore, the scaled restriction onto $Z$ is $(\partial_t+\pi ut\partial_t)|_Z=\partial_t$.
\end{proof}

We will not need the following example but include it to illustrate how the scaled derivation starts to deform when the total modulus does not vanish.

\begin{exam}
Assume that in the situation of the above lemma $\pi^2=0$ and the configuration is ordinary or Tate. Then the formal completion of $\cE_0$ along $\cP_0$ is the formal multiplicative group with coordinate $x$, hence the invariant vector field on the formal completion is $\partial_x-x\partial_x$ and the scaled restriction onto $Z$ (which only depends on the formal completion of $\cE_0$) equals $\pi^{-1}(\pi\partial_t-\pi^2t\partial_t)=\partial_t-\pi t\partial_t$.
\end{exam}

\subsubsection{Affine actions}
Let $S$ be a scheme and $E$ a prestable genus-1 $S$-curve provided with an action of a group $G$. We say that the action is {\em affine} if it induces a trivial action on $\Pic^0_{E/S}$ and preserves scaled vector fields in the following sense: the inclusion $\Stab_G(Z,\partial_Z)\subseteq\Stab_G(Z)$ is an isomorphism for any irreducible component $Z\subseteq E_s$ in a fiber of $E$. In other words, if $g\in G$ takes $Z$ to itself, then it also preserves the scaled derivation of $Z$.

\begin{rem}\label{affautrem}
We have chosen this somewhat clumsy working definition because a unique maximal subgroup of $\Aut_S(E)$ which acts affinely does not exist in general, though it does exist when the crown is connected. This is more relevant for the study of folding elliptic curves and will be worked out elsewhere.
\end{rem}

\begin{exam}\label{affexam}
Assume that $S=\Spec(R)$ is a fat point, $Z$ is a leaf with a coordinate $t$ such that the node $P$ of $Z$ is given by $t=\infty$, and the modulus of $P$ is 0. Set $\GG=\Aut(\PP^1,\infty)=\GG_a\rtimes\GG_m$, then $\GG(R)$ embeds into $\Aut_S(E)$ via its action on $(Z,P)$. By Lemma~\ref{scaleder2} any affine action on $E$ preserves $\partial_t$ on $Z$, and hence it acts on $Z$ through $\GG_a(R)$.
\end{exam}

\subsubsection{Affinity condition}
In the sequel by $(\ocM_{1,\Gamma})^{\Gamma,\flat}$ we denote the closed subscheme of $(\ocM_{1,\Gamma})^{\Gamma}$ defined by the condition that the action of $\Gamma$ is affine. We abbreviate this condition by ($\flat$).

\begin{rem}\label{flatrem}
(i) By Lemma~\ref{fixedlocus} $X(N)$ is the semistable Picard invariant component of $(\ocM_{1,\Gamma})^\Gamma\otimes\ZZ[1/N]$. We saw in Example~\ref{compexam} that there exist other components, and even classified them for $N=p>3$ in \S\ref{pinvsec}. All additional components correspond to non-affine $\Gamma$-actions -- not Picard invariant in the semistable case or through $\mu_N\subset\GG_m$ on a leaf, and hence $(\ocM_{1,\Gamma})^{\Gamma,\flat}\otimes\ZZ[1/N]=X(N)$. In particular, the semistability condition of Deligne-Rapoport imposed in the definition of generalized elliptic curves can be replaced by ($\flat$).

(ii) On the other hand, the affinity condition does not remove non-smoothable supersingular components from $\cX(p)$ because the action on the deep supersingular configurations is always affine, even when the core is ordinary. Thus on the level of sets $\cX(p)\otimes\FF_p\subsetneq(\ocM_{1,\Gamma})^{\Gamma,\flat}\otimes\FF_p=(\ocM_{1,\Gamma})^\Gamma\otimes\FF_p$.
\end{rem}

\subsection{Main results}
It remains to combine all definitions and computations we have already made together.

\subsubsection{The functor}\label{funsec}
First, we define $(\ocM_{1,\Gamma})^{\Gamma,\flat,\phi_p}$ to be the intersection of $(\ocM_{1,\Gamma})^{\Gamma,\flat}$ and $(\ocM_{1,\Gamma})^{\Gamma,\phi_p}$. This is the stack which classifies $\Gamma$-marked genus-1 curves satisfying both ($\phi_p$) and ($\flat$). We will prove that it coincides with $\cX(p)$.

\subsubsection{Deformations}
The key computational ingredient in the proof of our main results is to bound the tangent spaces of $(\ocM_{1,\Gamma})^{\Gamma,\flat,\phi_p}$.

\begin{lem}\label{mainlem}
Let $k$ be a field of characteristic $p$ and let $(E,\Gamma_k)$ be a $k$-point of $M=(\ocM_{1,\Gamma})^{\Gamma,\flat,\phi_p}$. Then the dimension $d$ of the tangent space $T_{(E,\Gamma_k)}(M)$ is at most two.
\end{lem}
\begin{proof}
Recall that $T=T_{(E,\Gamma_k)}((\ocM_{1,\Gamma})^{\Gamma})$ classifies deformations of a $\Gamma$-marked curve $(E,\Gamma_k)$ over $R=k[\veps]/(\veps^2)$, and we should compute the subspace of deformations that also satisfy conditions ($\flat$) and ($\phi_p$): the action is affine and the retraction of $\Gamma_k$ onto the core is a Drinfeld level structure. Also recall that by the exact sequence (\ref{basicseq}), $T$ contains the subspace $$K=\oplus_{i=0}^d T_{(Z_i,D_i)}(\ocM_{g_i,n_i})^{\Gamma_i/\Gamma_{i-1}}$$ of deformations which do not smoothen the nodes, and $\dim(T/K)$ is bounded by the number $d_2$ of types of nodes. Therefore $d\le d_1+d_2$, where $d_1$ is the dimension of the space $K^{\flat,\phi_p}=T_{(E,\Gamma_k)}(M)\cap K$ which classifies deformations as a $\Gamma$-marked curve which do not smoothen the nodes and satisfy conditions ($\flat$) and ($\phi_p$). We should prove that $d_1\le 1$ and it even vanishes in the Tate and canonical cases.

The affinity condition is componentwise, hence $K^\flat=T_{(E,\Gamma_k)}((\ocM_{1,\Gamma})^{\Gamma,\flat})\cap K$ is the direct sum of intersections. By Lemma~\ref{scaleder2}(i) the scaled vector field on rational components is of the form $\partial_t$, so as we observed in Example~\ref{affexam}(i) the action is affine if and only if it is through $\GG_a(R)$. It follows easily that the space is trivial for the action of $\ZZ/p\ZZ$ on $\PP^1$ and is one-dimensional (with the deformation parameter $\lam$) in the case of a faithful action of $\Gamma$ on $\PP^1$. This implies that in the ordinary case $d_1\le 1$. Also, in Tate's case we observe that the core components are just $\PP^1$'s with three nodes and the trivial action, hence they do not deform either, and $d_1=0$ in this case. So, we are already done with these two cases.

It remains to consider the case, when the configuration is supersingular and hence the core is $(E_0,O_k)$, where $E_0$ is a supersingular elliptic curve with the node at the origin $O_k$. We already know that $K^\flat=k^r\oplus T_{(E_0,O_k)}(\ocM_{1,1})$, where $r=0$ in the canonical case and $r=1$ in the deep case. In both cases, the dimension exceeds what we need by one, hence it suffices to prove that $K^{\flat,\phi_p}\subsetneq K^\flat$. Naturally, we will see that ($\phi_p$) forbids to deform the supersingular core, and hence kills the one-dimensional term $T_{(E_0,O_k)}(\ocM_{1,1})$ . Indeed, if $(\cE,\Gamma_R)\in K^{\flat,\phi_p}$ is a deformation of $(E,\Gamma_k)$, then the image of a section $\gamma_R\to\cE$ in the core $\cE_0$ coincides with $O_R$ because the node at $O_R$ is of modulus zero, and hence $\Gamma_R\to E_0$ is the zero section taken with multiplicity $p^2$. This implies that $\cE_0(R)$ has no non-trivial $p$-torsion, that is, $\cE_0$ is the trivial deformation $E_0\otimes_k R$ (or $j(\cE_0)=j(E_0)$ is a zero of $S_p$), as claimed.
\end{proof}

\subsubsection{Absolute deformations}
Since we have to study schemes which are not smooth over $\ZZ$, we should next refine Lemma~\ref{mainlem} to the absolute setting. We will see that the dimensions of the absolute tangent spaces are independent of cases, and this will also allow us as a by product to figure out what the precise dimensions of the relative tangent spaces are.

For a point $x$ on a scheme $M$ let $\tilT_x(M)$ denote the absolute tangent space, which is the dual of the absolute cotangent space $\tilT^*_x(M)=m_x/m_x^2$. If $p=\cha(k(x))$, then the cotangent space $T^*_x(M)=m_x/(p,m^2_x)$ over $\ZZ$ is the quotient of $\tilT^*_x(M)$ by the image of $p$, so we have an embedding $T_x(M)\into\tilT_x(M)$ whose cokernel is of dimension at most one and is non-zero if and only if there exists a non-trivial deformation over $\ZZ/p^2\ZZ$ -- the unramified case.

\begin{lem}\label{mainlem2}
Let $k$ be a field of characteristic $p$ and let $(E,\Gamma_k)$ be a $k$-point of $M=(\ocM_{1,\Gamma})^{\Gamma,\flat,\phi_p}$. Then the absolute tangent space $\tilT_{(E,\Gamma_k)}(M)$ is two-dimensional, and the embedding $T_{(E,\Gamma_k)}(M)\into\tilT_{(E,\Gamma_k)}(M)$ is an isomorphism if and only if the case is ramified: either $p>2$ or the configuration is supersingular.
\end{lem}
\begin{proof}
First we note that $M$ is a closed subscheme of $\ocM_{1,\Gamma}$ which contains $Y(p)$ because the two conditions defining $M$ hold on $Y(p)$ tautologically; in fact, they were designed as natural extensions of these tautological properties to the locus of non-smooth genus-1 curves. Thus, $\cX(p)$ is a closed subscheme of $M$ and the closed immersions $\cX(p)\into M\into (\ocM_{1,\Gamma})^{\Gamma,\phi_p}$ are bijective by Lemma~\ref{philem}(ii). In particular, $M$ is purely two-dimensional, and hence we should only prove that $\dim(\tilT_{x}(M))\le 2$, where $x=(E,\Gamma_k)$. We will use that $\dim(T_{x}(M))\le 2$ by Lemma~\ref{mainlem}.

First, consider the ramified case: either $p>2$ or the configuration is supersingular. We claim that in this case $T_{x}(M)=\tilT_{x}(M)$. It suffices to show that there are no deformations over $R=\ZZ/p^2\ZZ$. So, assume that $(\cE,\Gamma_R)$ is a $\Gamma$-marked deformation and let us show that the condition ($\phi_p$) is not satisfied. In fact, we claim that the generalized elliptic $R$-curve $\cE_0$ does not possess a Drinfeld level structure. Indeed, the connected component of $\cE_0[p]$ is given by the vanishing of a polynomial $\hatphi_p$ which annihilates the $p$-torsion in the corresponding formal group and is of degree $p^h$, where $h\in\{1,2\}$ is the height of the formal group, and $h=2$ if and only if $E_0$ is supersingular. If a Drinfeld level structure exists, $\hatphi_p(t)$ splits to a linear factors in $R[t]$. But $\hatphi_p(t)/t$ is irreducible because $\hatphi_p(t)$ equals $t^{p^h}$ modulo $p$ and equals $pt$ module $t^2$ (just lift $\hatphi_p$ with a factorization to $\ZZ[x]$ and use Eisenstein's criterion).

Assume now that $x$ is unramified: $p=2$ and the configuration is either Tate or ordinary. In this case, even the tangent space $T_x(M')$ to $M'=(\ocM_{1,\Gamma})^\Gamma$ is at most two-dimensional because the leaf does not contribute to the tangent space by Lemma~\ref{ratcomp}(i). Therefore, $\dim(\tilT_x(M'))\le 3$ and it suffices to prove that the embedding $\tilT_x(M)\into\tilT_x(M')$ is not an equality. We will do this by providing a $\Gamma$-equivariant deformation of $(E,\Gamma_k)$, with a non-affine action.

At this stage we have two possibilities. There exists such a deformation over $\FF_2$, but it smoothens the nodes and is more difficult to describe. So, let us use the advantage of working with the absolute tangent spaces and construct a deformation $(\cE,\Gamma_R)$ over $R=\ZZ/4\ZZ$ which does not smoothen the nodes and points in the direction of the Igusa component of diagram (\ref{types2}). Lift $E_0$ to a generalized elliptic $R$-curve $\cE_0$ and paste it with two leaves of modulus 0 along the $R$-sections $0_R,\tau_R\in\cE_0[2](R)$, where $\tau_R$ is an element which lifts the generator of $E_0[2](k)$. The action of $\Gamma=(\ZZ/2\ZZ)^2$ is lifted as follows: $\Gamma_0=\ZZ/2\ZZ$ acts on the leaves through an involution $\sigma\in \GG_a\rtimes\GG_m(R)$ which reduces to $t\mapsto t+1$ modulo 2 (e.g. $\sigma(t)=-t+1$) and is trivial on the core, and $\Gamma/\Gamma_0$ acts on the core $\cE_0$ via $\tau_R$. We claim that the action is not affine. By Lemma~\ref{scaleder2}(i) and Example~\ref{affexam}(i), any affine action on a leaf $\cZ$ should be through the group $\GG_a(R)=\ZZ/4\ZZ$. So, the generator of $\Gamma_0$ must act as $t\mapsto t+2$, but then the restriction to the underlying leaf $Z$ of $E$ would be trivial. A contradiction.
\end{proof}

%For completeness we include an example of a ``strange'' deformation in characteristic 2, which was mentioned in the above proof.

%\begin{exam}
%Let $R=\FF_2[\veps]/(\veps^2)$ and $\cE$ is glued from a $1$-gon $\cE_0$ with coordinate $t$ and a leaf $\cZ$ with coordinate $x$ so that $t(1+x)=\veps$. In particular, the node is of modulus $\veps$ and this is a deformation of the Tate configuration $E/\FF_2$ which smoothens the node. The action on $E$ is
%\end{exam}

\subsubsection{Local parameters}\label{paramsec}
Our description of deformations of $M=(\ocM_{1,\Gamma})^{\Gamma,\flat,\phi_p}$ can be made completely explicit and uniform for all primes by providing a natural family of parameters at two-dimensional regular rings $\cO_{M,x}$ with $x\in M\otimes\FF_p$ corresponding to a $\Gamma$-marked $k$-curve $(E,\Gamma_k)$. We will do this in terms of coordinates on components of $M\otimes\FF_p$ and smoothing parameters (or moduli) of each type of the nodes of $E$. The arguments are the same and we skip them: one inspects the proofs of Lemmas~\ref{mainlem} and \ref{mainlem2} and checks that the two parameters $s,t\in m_x$ are indeed linearly independent modulo $m_x^2$.

\begin{itemize}
\item[(o)] In all cases the $\FF_p$-fiber is given by the condition $(p)=(\pi_0^{p^2-1}\pi_1^{p-1})$ by Lemma~\ref{ramlem}.
\item[(i)] The canonical case. This configuration contains nodes of types $l_0,l_1$ and the smoothing parameters $s=\pi_0,t=\pi_1$ are regular parameters.
\item[(ii)] The Tate case. This configuration contains nodes of types $l_1,l$ and the smoothing parameters $s=\pi_1,t=\pi$ are regular parameters.
\item[(iii)] The ordinary case. Regular coordinates at $x$ are $\pi_1,j_x$, where $j_x$ is a lifting to $\cO_{M,x}$ of the regular parameter $j-j(E)\in\cO_{\Ig_p,x}$ of the Igusa curve through $x$.
\item[(iv)] The deep supersingular case. Regular coordinates are $\pi_0,\lam_x$, where $\lam_x$ is a lifting of the coordinate $\lam-\lam(E,\Gamma_k)$ on the $\PP^1_\lam$-component through $x$.
\end{itemize}

The situation when $p$ is invertible is the classical one: $M\otimes\ZZ[1/p]=X(p)$ and the parameter at $x\in M_\QQ$ is $j-j(E)$ in the smooth case and $\pi$ in the Tate case, while for $x\in M\otimes\FF_l$ with $l\neq p$ the second parameter is just $l$.

\subsubsection{Main theorem}
It remains to summarize our results about $\cX(p)$ and its comparison to the model of Katz-Mazur as follows.

\begin{theor}\label{mainth}
Let $p$ be a prime number and let $\cX(p)$ be the closure of $X(p)$ in $\ocM_{1,\Gamma}$, where $\Gamma=(\ZZ/p\ZZ)^2$. Then:

(i) The identity of $Y(p)$ extends to the isomorphism $\cX(p)=(\ocM_{1,\Gamma})^{\Gamma,\flat,\phi_p}$.

(ii) The scheme $\cX(p)$ is regular and $D(p)=\cX(p)\setminus Y(p)$ is an snc divisor whose irreducible components are as follows: $p+1$ horizontal Tate divisors at the cusps, $p+1$ vertical Igusa divisors and the supersingular vertical $\PP^1_\lam$ components.

(iii) The contraction of a $\Gamma$-marked curve onto its core gives rise to a generalized elliptic curve with Drinfeld level structure, and the induced morphism $\cX(p)\to\tilcX(p)$ is the blowup of Katz-Mazur model $\tilcX(p)$ at the union of the supersingular points.
\end{theor}
\begin{proof}
We already observed that the claim holds over $\ZZ[1/p]$. Thus, we will work only along the $\FF_p$-fiber. Recall that $\cX(p)$ is a dense closed subscheme of $M=(\ocM_{1,\Gamma})^{\Gamma,\flat,\phi_p}$ by Lemma~\ref{philem}(ii), and the absolute tangent spaces to $M$ are two-dimensional by Lemma~\ref{mainlem2}. Therefore, $M$ is regular and hence coincides with its reduction $\cX(p)$. In addition, $D(p)$ is locally given by the vanishing of the product of parameters corresponding to the types of nodes, hence it is snc.

It remains to prove (iii). Note that $\cX(p)\to\tilcX(p)$ is a morphism between regular proper $\ZZ$-schemes, which restricts to the identity on $X(p)$ and contracts the supersingular $\PP^1_\lam$ components. Of course, no Igusa component is contracted. Since any birational modification of regular surfaces is a composition of blowups at closed points, $\cX(p)$ is the blowing up of $\tilcX(p)$ at the union of all supersingular points over $\FF_p$.
\end{proof}

\begin{rem}
In the same way (but much faster) one can check that the affinity condition precisely removes the non-smoothable component (without leaving an embedded component along the intersection) from the moduli space $(\ocM_{1,\Gamma_0})^{\Gamma_0}$ of $\Gamma_0$-marked elliptic curves, and hence $(\ocM_{1,\Gamma_0})^{\Gamma_0,\flat}=\cX_1(p)$ coincides with the model of Katz-Mazur. In this case, $\Gamma_0$-stabilization of the universal curve does not require to modify the moduli space.
\end{rem}

\subsection{Stable reduction}
The results of this subsection are very special for level $p$. They do not hold for $\cX(p^n)$ with $n\ge 2$, and finding a stable reduction for the latter schemes is a much more delicate problem, which was studied in the beginning of 2000s by Wewers, Weinstien and others, and was finally solved by Weinstein in \cite{Wei}.

\begin{theor}\label{stableth}
Let $p$ be a prime number and set $D(p)=\cX(p)\setminus Y(p)$. Then

(i) $(\cX(p),D(p))\to(\Spec(\ZZ),\Spec(\FF_p))$ is a log smooth morphism of log regular schemes.

(ii) A finite extension $K/\QQ(\mu_p)$ with ring of integers $R=\cO_K$ gives rise to a stable reduction of $\cX(p)$ if and only if all primes over $(1-\xi_p)$ have ramification index divisible by $p+1$. In this case the stable model is the normalization $\cX(p)_R$ of $\cX(p)\otimes_{\ZZ[\xi_p]}R$.

(iii) Take $R=\cO_K$ of degree $p+1$ over $\ZZ[\xi_p]$ and totally ramified over $1-\xi_p$. For example, one can take $R=\ZZ[\xi_p,\pi]/(\pi^{p+1}-1-\xi_p)$. Then the morphism $\cX(p)_R\to\cX(p)$ is totally ramified over the Igusa curves in the $\FF_p$-fiber, while each supersingular component $Z=\PP^1_{\lam,\FF_p}$ is covered in $\cX(p)_R\otimes\FF_p$ by the Drinfeld's curve $Z'$ given by $t^{p+1}=\lam^p-\lam$.
\end{theor}
\begin{proof}
Claims (i) and (ii) follow from generalities about semistable reduction in the tame case. First, (i) holds because locally the morphism is given by $p=u\pi_0^{p^2-1}\pi_1^{p-1}$, where $u$ is a unit, both $\pi_0$ and $\pi_1$ are either units or regular parameters, and the exponents are not divisible by $p$. Then it follows from (i) that the stable reduction is attained over any extension of $\ZZ$ which is ramified over $p$ of order $N$ divisible by $p^2-1$. Recall that $\cX(p)$ (and already $\tilcX(p)$) is defined over $\ZZ[\xi_p]$, which is totally ramified over $p$ of order $p-1$. Hence, the situation over $(1-\xi_p)$ is as follows: $\cX(p)$ is smooth along the Igusa locus and has fibers of multiplicity $p+1$ along the supersingular locus. Therefore stable reduction is achieved after any additional extension $R/\ZZ[\xi_p]$ which is ramified at any prime over $(1-\xi_p)$ with the order of ramification divisible by $p+1$.

(iii) The total ramification over the Igusa components of $\cX(p)$ is clear, since they are of multiplicity one over $(1-\xi_p)$. Let $Z'$ be the preimage of a component $Z=\PP^1_{\lam,\FF_p}$ in $\cX(p)_R$. The morphism $Z'\to Z$ is totally ramified over the $p+1$ points of $\PP^1(\FF_p)$ which are the intersections with the Igusa components. Also, it is \'etale over the rest of $Z$ because $(\pi_0^{p+1})=(1-\xi_p)$ over any point in the deep supersingular locus.

The group $\GL_2(\FF_p)=\Aut(\Gamma)$ acts on $\cX(p)$ through the action on the group $\Gamma$ of marked points. In fact, it is the classical Galois group of the \'etale cover $X(p)\to X(1)=\PP^1_{j,\ZZ}$. Since $X(p)$ is defined over $\ZZ[\xi_p]$ one obtains a restriction homomorphism $\GL_2(\FF_p)\to\Gal(\QQ(\xi_p)/\QQ)=\FF^\times_p$, which is the determinant map and has $\SL_2(\FF_p)$ as the kernel.

The action on $\cX(p)$ preserves $Z$ and hence the cover $h\:Z'\to Z$ is $\SL_2(\FF_p)$-equivariant. In particular, it is $\FF_p$-equivariant, where $G=\FF_p$ acts by $\lam\mapsto\lam+1$. It follows that $h$ is the pullback of a morphism $h/G\:Z'/G\to Z/G$, where $Z/G$ is the projective line with the coordinate $\lam^p-\lam$ and $h/G$ is of degree $p+1$ and ramifies only over $0,\infty$. Thus, $Z'/G$ is the projective line with the coordinate $t$ given by $t^{p+1}=\lam^p-\lam$, and hence the pullback $Z'\to Z$ is as claimed.
\end{proof}

\appendix

\section{Prestable curves}
Our conventiones are as in \S\ref{convsec}.

\subsection{Geometric curves}
First, let us very briefly summarize some facts about prestable marked curves $(C,D)$ over an algebraically closed field $k$. Our main goal is to fix the terminology, which is not always standard, and choose the presentation line which will naturally extend to general bases. We assume that $2p_a(C)+|D|\ge 2$, so $(C,D)$ possesses a semistabilization, as we recall below.

\subsubsection{Stability conditions}
If $Z$ is an irreducible component of $C$ let $n(Z)$ be the cardinality of the preimage of $C^\sing\cup D$ in the normalization of $Z$. Recall that $\max(0,3-n(Z)-2g_Z)$ is the dimension of the image of the homomorphism of algebraic groups $\Aut(C,D)\to\Aut(Z)$. We say that $Z$ is {\em stable}, {\em semistable} or a {\em leaf} if $2g_Z+n(Z)\ge 3$, $2g_Z+n(Z)\ge 2$ or $2g_Z+n(Z)=1$, respectively. A marked curve $(C,D)$ is called {\em stable} or {\em semistable} if all its components are.

\subsubsection{Semistabilization and the core}
Each leaf $Z$ is a smooth rational curve without marked points and with a single node $P_Z$ such that the other component through $P_Z$ is semistable. Typically, we choose a coordinate $t_Z$ on $Z$ so that the node is at infinity. Contracting all leaves onto their nodes and keeping $D$ unchanged we obtain a retraction $C\to C'$ onto the union of semistable components. Some of these components can get destabilized in $C'$ and iterating the process we obtain the canonical {\em semistabilization} procedure $C=C^0\to C^1\to\dots C^n=C^\infty$, where $(C^\infty,D)$ is the maximal semistable marked subcurve, that will be called the {\em core} of $(C,D)$. It is obtained from $(C,D)$ by contracting all $\PP^1$-trees $T$ without marked points and with a single boundary node $\{P\}=T\cap\overline{C\setminus T}$.

\begin{rem}
The semistabilization is synchronized (or determined) by the {\em layer function} $\ell\:C^\gen\to\NN\cup\{\infty\}$ on the set of generic points: $\eta\in C^{\ell(\eta)}\setminus C^{\ell(\eta)+1}$ if $\eta$ is not in the core, and $\ell(\eta)$ is the minimal number such that $C^{\ell(\eta)}=C^\infty$ otherwise.
\end{rem}

\subsubsection{The crown}\label{layersec0}
The {\em crown} $\brC$ of $(C,D)$ is defined to be the smooth locus of the lowest layer. If $(C,D)$ is semistable, then $\brC=C^\sm$ is the complement to the nodal set, and otherwise $\brC=C^0\setminus C^1=C^0\cap C^\sm$ is the disjoint union of the non-nodal parts of the leaves.

\subsubsection{Stabilization}\label{geomstabsec}
If $(C,D)$ is semistable and $p_a(C)+n\ge 3$, one can further contract all semistable components: any maximal chain of $\PP^1$'s without marked points is contracted to a single node, and any maximal chain of $\PP^1$'s whose first component is marked by a point is contracted to a marked point. Thus, $C\to C^\st$ is birational on stable components and contracts the rest.

\subsubsection{Foldings}
More generally, by a {\em folding} of $(C,D)$ we mean a morphism of prestable curves $(C,D)\to(C',D')$ such that $C\to C'$ is either birational on a component or folds it (i.e. retracts to a point) and the divisor and genus are preserved: $D=D'$ and $p_a(C)=p_a(C')$. It is easy to see that $(C^\st,D)$ is the final folding of $(C,D)$. Other examples are provided by partial semistabilizations $C\to C^i$.

\subsubsection{Logarithmic forms and vector fields}
Let $\omega_{(C,D)}=\omega_C(D)$ denote the dualizing sheaf $\omega_C$ of $C$ twisted by $D$. One usually identifies $\omega_{(C,D)}$ with the sheaf of differential forms on the normalization $\tilC$ with at most simple poles at the preimages of the nodes and marked points so that the sum of residues over each node is zero. Similarly, its dual $\cT_{(C,D)}=\cT_C(-D)$ can be interpreted as the sheaf of vector fields on $\tilC$ with zeros at the nodes and marked points and a similar compatibility condition over the nodes.

\begin{lem}\label{geomH0lem}
Let $(C,D)$ be a prestable $n$-marked curve with $2p_a(C)+n\ge 2$.

(i) Assume that either $d=1$ or $d\ge 1$ and $\brC$ is connected. Then any section of $\omega^{\otimes d}_{(C,D)}$ with $d\ge 1$ vanishes outside of $C^\infty$ yielding the restriction isomorphism $H^0(\omega^{\otimes d}_{(C,D)})=H^0(\omega^{\otimes d}_{(C^\infty,D)})$.

(ii) Any section of $\cT_{(C,D)}$ vanishes outside of $\brC$. If $(C,D)$ is not semistable and $\{Z_i\}_{i\in I}$ is the set of leaves with coordinates $t_i$, then $H^0(\cT{(C,D)})=k^I$ and its basis is formed by the vector fields $\partial_{t_i}$ supported on a single leaf. If $(C,D)$ is semistable, then either $2p_a(C)+n=2$ and then $\cT_{(C,D)}=\cO_C$, or $2p_a(C)+n>2$ and then $H^0(\cT_{(C,D)})=0$.
\end{lem}
\begin{proof}
(i) Any section of $\omega_{(C,D)}$ vanishes on any leaf $Z$ by the degree considerations and hence also does not have a pole from the other side of the node of $Z$. This allows us to retract the leaves without affecting $H^0(\omega_{(C,D)})$ and the claim follows by induction.

Assume now that the crown is connected and $\phi\in H^0(\omega^{\otimes d}_{(C,D)})$. Then the complement of $C^\infty$ forms a chain $Z_0,Z_1\.Z_m$ of $\PP^1$'s, and by degree considerations $\phi|_{Z_i}=c_i\frac{(dt_i)^{\otimes d}}{t_i^d}$ and $c_0=0$. Using compatibility at the nodes we see that in fact $0=c_0=c_1=\dots =c_m$.

(ii) The claim is true when $C$ is elliptic and $D=\emptyset$, so we exclude this case. Let $\partial$ be a logarithmic vector field and $Z$ an irreducible component. If $Z$ is stable, then $\partial|_Z=0$ by degree considerations and if $Z$ is semistable but not stable, then it is rational with a coordinate $t$ such that $0,\infty$ lie in the union of the nodal and marked sets, and hence $\partial|_Z=ct\partial t$. Furthermore, the value of $c$ propagates along a chain of semistable components, hence a chain with a non-zero $c$ cannot start in the stable locus. It follows that if $\partial|_Z\neq 0$ for a semistable component, then $C$ is either an $m$-gon of $\PP^1$'s or a chain of $\PP^1$'s with one marked point on the tail and on the head. In either case $2p_a(C)+n=2$ and $\omega_{(C,D)}=\cO_C$. It the remaining case $\partial$ vanishes outside of the crown $\brC=\coprod_iZ_i$ and each $\partial|_{Z_i}$ is a vector field which does not propagate to the next component and hence has a double zero at $\infty$. This only leaves the possibility of $\partial=\sum_i c_i\partial_{t_i}$.
\end{proof}

\subsubsection{The genus-1 case}\label{g1sec}
In the paper we work with prestable curves $E$ of genus 1. In such a case the semistable core $E^\infty$ is either a smooth genus-1 curve or a chain of $n$ components $\PP^1_k$ that we call an {\em $n$-gon} and provide with coordinates so that zeros are glued with infinities. This includes the case of a 1-gon, when zero and infinity on the same component are glued to a node. The remaining part of $E$ is a disjoint union of rational trees with a single boundary node on $E^\infty$. We orient it from the leaves to the boundary node and sometimes filter by layers (e.g. in figure (\ref{types3})). Typically, we will reserve the notation $E_0$ for the core and other components will be denoted $Z_i$.

\subsection{Arbitrary base}
Now let us work with a prestable marked curves $(C,D)$ over an arbitrary base $S$ and let $f\:C\to S$ denote the structure morphism. Most of our results about foldings can be deduced by use of Knudsen's stabilization maps, though we suggest a short and self-contained route.

\subsubsection{The set of components}
Let $(C/S)^\gen$ denote the set of generic points of the $S$-fibers of $C$, which we also identify with the set of irreducible components of the $S$-fibers. If $s$ specializes $t$ in $S$, then any irreducible component of $C_s$ has a unique generization in $C_t$, and any component of $C_t$ has at least one specialization in $C_s$.

\subsubsection{The modulus}
If $P$ is a node in $C_s$, then \'etale locally $\cO_P=\cO_s[x,y]/(xy-\pi)$, where $\pi\in\cO_s$ is uniquely determined up to a unit and will be called the {\em modulus} of the node over $s$. A good intuition is provided by the particular case, when $S=\Spec(R)$ for a real valuation ring $R=\cO_s$, and hence the preimage of $P$ in the analytification of $C$ is an open annulus of modulus $\pi$. Moduli are tightly related to the natural log structure on $C$ and provide essentially the same information.

\subsubsection{The crown and the core}\label{crownsec}
We define the {\em crown} and the {\em core} of $(C,D)$ to be the fiberwise unions $\brC(D)=\coprod_{s\in S}\brC_s(D_s)$ and $C^\infty(D)=\coprod_{s\in S}C^\infty_s(D_s)$. The dependency on $D$ will often be omitted from the notation. Using the specialization relation one checks that the crown $\brC$ is an open $S$-scheme subscheme, while the core (so far) is just a constructible subset with a locally closed complement. This set will acquire a reasonable scheme structure as a folding (or contraction) of $C$.

\subsubsection{Smoothability}
Any $n$-pointed prestable curve of genus $g$ is (locally) {\em smoothable} in the following sense:

\begin{lem}\label{smoothabilitylem}
If $(C,D)$ is an $n$-pointed prestable curve of genus $g$ over a base $S$, then there exists an \'etale cover $S'\to S$ and a closed embedding $S'\into T$ such that $T$ is a disjoint union of integral schemes and the base change $(C',D')$ extends to a prestable $n$-pointed $T$-curve with a smooth generic fiber. Moreover, the family of smoothings is filtered.
\end{lem}
\begin{proof}
After passing to a cover $S''\to S$ we can add smooth sections so that the curve becomes $n'$-stable and a map $S''\to\calM_{g,n'}$ arises. Choose an \'etale cover $\coprod_{i=1}^lM_i\to\calM_{g,n'}$ with affine integral $M_i$ and set $S'_i=S''\times_{\calM_{g,n'}}M_i$ and $S'=\coprod_i S'_i$. We pullback the tautological families $S'$ and remove the additional $n'-n$ sections, obtaining a prestable $n$-pointed curve on $M$ with smooth generic fiber, which pulls back to $(C',D')=(C,D)\times_SS'$. By a further localization we can assume that each $S'_i$ is affine, and then the map $S'_i\to M_i$ factors through a closed immersion into an integral scheme $T_i$ and we can set $T=\coprod T_i$.

As for cofinality, if there exists two smoothings, then enlarging the covers of $S$ we can assume that both exist over the same $S'$ say $S'\into T$ and $S'\into T'$ such that $(C',D')$ extends to generically smoooth $(C_T,D_T)$ and $(C_{T'},D_{T'})$. Consider the pinchings $T''=T\coprod_{S'}T'$ and $C_{T''}=C_T\coprod_{C'}C_{T'}$, and similarly for the divisors. Then $(C_{T''},D_{T''})$ is a prestable $n$-pointed curve over $T''$, and we can apply the above result to smoothen it. This induces a smoothing of $(C,D)$ over an \'etale cover $S''$ of $S'$ through which the two smoothings we started with factor (after the base change to $S''$).
\end{proof}

In the same manner, we say that an $S$-automorphism of $(C,D)$ or a group structure is {\em smoothable} if up to an \'etale cover of $S$ it can be lifted to an automorphism or a group structure of a family with a smooth generic fiber.

\subsubsection{Foldings and stabilization}
By a {\em folding} of $(C,D)$ we mean a morphism $(C,D)\to(C',D)$ of prestable $S$-curves which is a fiberwise folding. In the same spirit, a prestable marked $S$-curve $(C,D)$ is {\em semistable} or {\em stable} if all its geometric fibers are. Note that a component $Z$ in a fiber $C_s$ is a leaf, semistable or stable if and only if $\deg(\omega|_Z)$ is negative, zero or positive, respectively. It follows that $(C,D)$ is semistable if and only if $\omega=\omega_{(C,D}$ is semi-ample, and in this case $(C,D)$ possesses a {\em stabilization map} $(C,D)\to(C^\st,D^\st)$, where $$C^\st=\Proj_S\left(\oplus_d f_*\omega_{(C,D)}^{\otimes d}\right)$$ and $D^\st$ is the image of $D$. In fact, this is the relative generalization of the ad hoc stabilization construction from \S\ref{geomstabsec}.

\begin{lem}\label{stablem}
Assume that $(C,D)$ is a semistable marked $S$-curve such that $g+2n\ge 3$, then

(i) The stabilization map $h\:(C,D)\to(C^\st,D^\st)$ is a folding with a stable image.

(ii) Any folding $g\:C\to C'$ which contracts the components folded by $h$ factors through $h$. In particular, $h$ is unique (up to a unique isomorphism).

(iii) The stabilization construction is compatible with any base change $S'\to S$.

(iv) Assume that $s\in S$ is a point and $P$ is a node in the fiber $C^\st_s$. Then the modulus of $P$ equals the product of moduli of all nodes in the preimage of $P$ in $C$.
\end{lem}
\begin{proof}
(iii) It suffices to prove that the sheaves $f_*(\omega^{\otimes d}_{(C,D)})$ are locally free (and hence flat), and this is done by Grauert's theory of direct images, using that $H^1(C_s,\omega_{(C_s,D_s)})$ equals $k$ when $D$ is empty and vanishes otherwise.

(i) Since $C^\st$ is $S$-flat by (iii), we can work fiberwise and then it is easy to see that the map contracts the semistable components and does not modify the stable ones.

(ii) Note that $C\to C''=\Proj_{C'}\left(\oplus_d g_*\omega_{(C,D)}^{\otimes d}\right)$ factors through $h$. Then working fiberwise one can easily check that the morphism $C\to C''$ folds the same components as $g$.

(iv) This can be checked locally, hence we can use the smoothing technique from Lemma~\ref{smoothabilitylem} and reduce to the case of an integral local $(S,s)$. Furthermore, this equality can be checked in a valuation ring dominating $\cO_s$, hence we can replace $S$ by the spectrum of a DVR dominating it. Now we can consider the non-archimedean analytification $\lam\:C^\an\to C_s$ and use that $P\in C_s$ is a node of modulus $\pi$ if and only if $\lam^{-1}(P)$ is an open annulus of modulus $|\pi|$. From this point of view our claim is just a reinterpretation of the fact that the modulus of an annulus is multiplicative with respect to subdivisions of an annulus into a union of annuli (two annuli on the sides are semiclosed and the rest are closed).
\end{proof}

\subsubsection{Uniqueness of foldings}
Using the above result and the standard trick with stabilization by increasing $D$ we can now establish uniqueness of foldings.

\begin{lem}\label{uniquefold}
A folding $h\:(C,D)\to(C',D)$ is determined uniquely by the set $\cZ\subset(C/S)^\gen$ of components folded by $h$.
\end{lem}
\begin{proof}
The claim is local on $S$, so we can assume that there exists a large set $D'\supseteq D$ of sections of $C'^\sm$ which stabilize $C'$, that is, $(C',D')$ is stable. Note that $C\to C'$ is an isomorphisms over $C'^\sm\setminus D$ and hence $D'$ uniquely lifts to $C$. By an additional localization we can also stabilize $(C,D')$ by increasing $D'$ further to $D'_n=\{D',P_1\.P_n\}$. Set also $D'_i=\{D',P_1\.P_i\}$. Removing one marked point from a stable pair produces a semistable pair, hence Lemma~\ref{stablem} applies iteratively and we obtain a sequence of foldings 
$$C=C_n\to C_{n-1}\to\ldots\to C_0=C'$$ such that each $(C'_i,D'_i)$ is stable. In addition, Lemma~\ref{stablem}(ii) implies that any folding $(C,D)\to(C'',D)$ which folds $\cZ$ factors through this sequence. Therefore we obtain a folding $(C',D)\to(C'',D)$, which induces isomorphisms on fibers and hence is an isomorphism.
\end{proof}

\subsubsection{Semistabilization}
The same trick can be used to construct foldings. Since we have already constructed stabilization of semistable curves, let us describe partial semistabilizations which do not fold components of the core. We say that a set $\cZ\subset(C/S)^\gen$ of irreducible components in the fibers is {\em foldable} if there exists a folding $(C,D)\to(C',D)$ which folds precisely the components from $\cZ$.

\begin{lem}\label{contrlem}
Let $(C,D)$ be a prestable marked $S$-curve and let $\cZ\subset(C/S)^\gen$ be a set of components disjoint from the core and from $D$. Then $\cZ$ is foldable if and only if the following conditions are satisfied:
\begin{itemize}
\item[(i)] Each $\cZ_s$ is foldable in the fiber $(C_s,D_s)$. Equivalently, $\cZ_s$ is a disjoint union of rational trees having a single node joint with its complement.
\item[(ii)] The set $\cZ$ is closed under specialization.
\item[(iii)] If $Z\in C_t^\gen\setminus\cZ$ and $s$ is a specialization of $t$, then $Z$ has a specialization in $C_s^\gen\setminus\cZ$.
\end{itemize}
\end{lem}
\begin{proof}
If $\cZ$ is the set of components folded by a folding $C\to C'$, then it is easy to see that (i)--(iii) hold. Conversely, let us assume that these conditions are satisfied and let us construct the folding $h\:C\to C'$. By uniqueness of foldings, it suffices to construct $h$ locally on $S$, so we can assume that as many sections as we need exist. The condition on $\cZ$ is designed so that there exists an enlargement $D'$ of $D$ which avoids the components of $\cZ$ and stabilizes all other components in the fibers. Then as in the proof of Lemma~\ref{uniquefold} we enlarge $D'$ to some $D'_n=\{D',P_1\.P_n\}$ which stabilizes the whole $C$, and then iteratively remove $P_i$'s and stabilize.
\end{proof}

Since the set of components outside of the core clearly satisfies all assumptions of the lemma we obtain the following corollary.

\begin{cor}\label{semistcor}
Let $(C,D)$ be a prestable marked $S$-curve. Then there exits the semistabilization folding $C\to C^\infty$.
\end{cor}

\begin{rem}\label{semistablerem}
Set-theoretically $C^\infty=C^\infty(D)$ is embedded in $C$, but this does not respect the schematic structure (unless $S=\Spec(k)$). From now on we view the core $C^\infty$ as a folding of $C$, with a set-theoretic section $C^\infty\into C$.

%(ii) Despite some choices we have managed to construct the canonical folding onto the core. One can wonder if a more canonical construction exists, but the probable answer is negative. A typical bad example is in the style of Hironaka's famous threefold -- take $S$ a surface with closed points $s,t$, take any prestable $C$ over $S$ and find two sections $P,Q$ which intersect transversally over $s$ and $t$.
\end{rem}

\subsubsection{The local base case}\label{localfold}
Over a local base $(S,s)$ one has a canonical factorization of the morphism $C\to C^\infty$ which extends that of $C_s\to C_s^\infty$. Namely, let $\cZ$ be the set of components that specialize only to a leaf in $C_s$. Then $\cZ$ is foldable by Lemma~\ref{contrlem}, we obtain a folding $C\to C^1$, and proceed by induction. This construction does not globalize over $S$ and is not even compatible with localizations at points $t\in S$. The layer function $\ell\:(C/S)^\gen\to\NN$ induced by this factorization is the minimal function which is non-decreasing with respect to specializations and dominates all fiberwise layer functions $\ell_t\:C_t^\gen\to\NN$.

\subsubsection{Localization along a component}\label{stabcompsec}
We will also use the following construction: if $Z$ is a geometrically connected component in $C_s$ not lying in the core, one can add a section $O_S$ through $Z$, set $D'=\{D,O_S\}$ and fold $C$ onto $C'=C^\infty(D')$. This folds all components in $C_s$ except those in $C_s^\infty$ and in the chain connecting $Z$ to the core. In particular, $C'$ has at most one leaf in $C_s$ and hence in each other fiber $C_t$.

%\begin{exam}\label{stabcompexam}
%(i) If $E$ is a genus-1 prestable $S$-curve
%\end{exam}

\subsection{Deformations of stable curves}
Deformations of stable curves is a very standard material, so we just formulate a few results we need.

\subsubsection{Notation}
Let $(C,D)$ be a stable $n$-pointed genus-$g$ curve over a field $k$. By $P=\{P_j\}_{j\in J}$ and $\{\tilC_i\}_{i\in I}$ we denote the sets of nodes and normalized irreducible components of $C$. Also, let $\tilD_i$ be the preimage of $P\cup D$ in $\tilC_i$ and $n_i=|\tilD_i|$. We will now recall the classical computation of the tangent space $T_{(C,D)}(\ocM_{g,n})$ at the point $(C,D)$. By our convention, this is the relative tangent space over $\ZZ$ at the image $x\in\ocM_{g,n}$ of $(C,D)$, which is the same as the tangent space to $\ocM_{g,n}\otimes\FF_p$ at $x$.

\subsubsection{The tangent space}
The space of first order deformations is computed as $T_{(C,D)}(\ocM_{g,n})=\Ext^1_{\cO_C}(\Omega^1_C(D),\cO_C)$, where the twist detects the freedom to move $D$ in $\ocM_{g,n}$. Using the local-to-global spectral sequence computing $\Ext$ we split this group between the two groups $H^i(C,\calExt^{1-i}(\Omega^1_C(D),\cO_C))$, with $i=1$ corresponding to the global part and $i=0$ describing local deformations of the nodal singularities.

\subsubsection{Nodal contributions}
Each node has a one-dimensional deformation space corresponding to the smoothing parameter $\pi_j$ in the local equation $xy=\pi_j$ and it contributes the term $T_{P_j}=T_{P'_j}(\tilC'_j)\otimes T_{P''_j}(\tilC''_j)\toisom k$, where $P'_j$ and $P''_j$ are the preimages of $P_j$ in the normalization $\tilC$, and $\tilC'_j$, $\tilC''_j$ are the components (possibly the same component) containing them.

\subsubsection{Contribution of components}
The global term describes the part of the deformation obtained when the nodes do not smoothen infinitesimally (i.e. $\pi_j=0$). It splits to the direct sum of deformations of $(\tilC_i,\tilD_i)$, which signals that both nodes and marked points are allowed to travel under the deformation.

\subsubsection{The summary}\label{sumapp}
Tying everything together we simplify the final formula as
\begin{equation}
0\to\oplus_{i\in I} T_{(\tilC_i,\tilD_i)}(\ocM_{g_i,n_i})\to T_{(C,D)}(\ocM_{g,n})\to\oplus_{j\in J} T_{P_j}\to 0,
\end{equation}
where
\begin{equation}
T_{(\tilC_i,\tilD_i)}(\ocM_{g_i,n_i})=H^1(\tilC_i,T_{\tilC_i}(-\tilD_i))=H^0(\tilC_i,\Omega^{\otimes 2}_{\tilC_i}(\tilD_i))'
\end{equation}
 because in this case $\Omega_{\tilC_i}(\tilD_i)$ is locally free.

\subsection{Equivariant deformations}
Finally, the results about equivariant deformations follow from the non-equivariant analogs.

\subsubsection{The action and the fixed locus}
Assume now that a finite group $\Gamma$ acts on $\{1\.n\}$ and hence also on $\ocM_{g,n}$ by permuting the sections. Then the fixed locus $\ocM_{g,n}^\Gamma$ is defined and it is a closed substack, see \S\ref{fixedsec}.

\subsubsection{$\Gamma$-equivariant tangent space}
Now, assume that $(C,D)$ is a $\Gamma$-equivariant point of $\ocM_{g,n}$ and let us compute the tangent space $T_{(C,D)}(\ocM_{g,n}^\Gamma)$. Since the space $H^0(C,T_C(D))$ vanishes by the stability of $(C,D)$, we have that
\begin{equation}
T_{(C,D)}(\ocM_{g,n}^\Gamma)=H^1(C,T_C(D))^\Gamma=T_{(C,D)}(\ocM_{g,n})^\Gamma.
\end{equation}
Thus we should just apply the functor of $\Gamma$-invariants to the sequence computing $T_{(C,D)}(\ocM_{g,n})$ in \S\ref{sumapp}.

\subsubsection{The summary}\label{summary}
For any node $P_j$ its stabilizer $\Gamma_j$ acts on $T_{P_j}$, and an element $(t_j)_{j\in J}\in\oplus_{j\in J} T_{P_j}$ is $\Gamma$-invariant if and only if each $t_j$ is $\Gamma_j$-invariant and also $t_j=t_{j'}$ whenever $j$ and $j'$ lie in the same $\Gamma$-orbit of $J$. Thus it suffices to pick up a single representative in any $\Gamma$-orbit: $(\oplus_{j\in J} T_{P_j})^\Gamma\toisom\oplus_{j\in J/\Gamma} (T_{P_j})^{\Gamma_j}$. Absolutely the same reasoning applies to the components $C_i$ and their stabilizers $\Gamma_i$ and we obtain the exact sequence
\begin{multline}\label{summaryeq}
0\to\oplus_{i\in I/\Gamma} T_{(\tilC_i,\tilD_i)}(\ocM_{g_i,d_i})^{\Gamma_i}\to T_{(C,D)}(\ocM_{g,d})^\Gamma\to\oplus_{j\in J/\Gamma} (T_{P_j})^{\Gamma_j}\stackrel{\partial}{\to}\\ \oplus_{i\in I/\Gamma} H^1(\Gamma_i,T_{(\tilC_i,\tilD_i)}(\ocM_{g_i,d_i}))\to\ldots
\end{multline}

If $|\Gamma|$ is not invertible the $H^1(\Gamma_i,\cdot)$ term (and the boundary map) can be non-zero, providing an obstacle for equivariant smoothing of the nodes. Besides this complication, the equivariant deformation space is combined from the terms $(T_{P_j})^{\Gamma_j}$ and $$T_{(\tilC_i,\tilD_i)}(\ocM_{g_i,d_i})^{\Gamma_i}=H^1(\tilC_i,T_{\tilC_i}(-\tilD_i))^{\Gamma_i}.$$ Each space $(T_{P_j})^{\Gamma_j}$ is either trivial or one-dimensional, and if $\Gamma_j$ does not switch the branches at $P_j$, then the second possibility occurs if and only if it acts on $T_{P'_j}(\tilC_{j'})$ and $T_{P''_j}(\tilC_{j''})$ through opposite characters $\Gamma_j\to k^\times$.

\bibliographystyle{amsalpha}
\bibliography{wild}

\end{document}